\documentclass{article}

\usepackage[a4paper,bottom=2cm,top=2cm,left=2.54cm,right=2.54cm]{geometry}
\usepackage[english]{babel}
\usepackage[utf8]{inputenc}
\usepackage{amsmath}
\usepackage{amssymb}
\usepackage{amsthm}
\usepackage{color}
\usepackage{soul}
\usepackage{graphicx}
\usepackage{comment}

\newcommand{\defeq}{:=}
\DeclareMathOperator*{\argmin}{arg\,min}

\DeclareMathOperator*{\supp}{supp}

\newcommand{\norm}[1]{\|#1\|}

\newcommand{\abs}[1]{|#1|}

\newcommand{\grad}[1]{\bm{\nabla} #1}

\DeclareMathOperator{\BV}{BV}
\DeclareMathOperator{\BD}{BD}
\DeclareMathOperator{\TV}{TV}
\DeclareMathOperator{\TGV}{TGV}
\DeclareMathOperator{\TVL}{TVL}
\DeclareMathOperator{\TVLp}{\TVL^p}
\newcommand{\TVpwL}{\TV_\text{pwL}}

\let\div\relax
\DeclareMathOperator{\div}{\divergence}

\DeclareMathOperator{\dist}{dist}

\DeclareMathOperator{\PSNR}{PSNR}
\DeclareMathOperator{\SSIM}{SSIM}
\newcommand{\weakto}{\mathrel{\rightharpoonup}}

\newcommand{\R}{\mathbb R}

\newcommand{\charf}{\chi}

\newcommand{\Mg}{\mathfrak{M}}
\newcommand{\M}{\mathcal{M}}

\renewcommand{\leq}{\leqslant}
\renewcommand{\geq}{\geqslant}
\renewcommand{\phi}{\varphi}

\newcommand{\reg}{\mathcal J}

\newcommand{\C}{\mathcal{C}}

\newtheorem{theorem}{Theorem}
\newtheorem{corollary}[theorem]{Corollary}
\newtheorem{lemma}[theorem]{Lemma}
\newtheorem{proposition}[theorem]{Proposition}

\theoremstyle{definition}
\newtheorem{definition}[theorem]{Definition}

\newtheorem*{assumption*}{Assumption}
\newtheorem{remark}[theorem]{Remark}
\newtheorem*{remark*}{Remark}
\newtheorem*{definition*}{Definition}

\providecommand{\keywords}[1]{\textbf{Keywords: } #1}
\providecommand{\ams}[1]{\textbf{AMS subject classifications: } #1}

\graphicspath{{../pic/}}

\usepackage[textsize=footnotesize,colorinlistoftodos]{todonotes}

 \usepackage{bm}
 \usepackage{mathtools}
 \usepackage{xcolor}
 \usepackage{tabularx}
 \usepackage{booktabs}
\usepackage{multirow}
\usepackage{csquotes}

\usepackage[bookmarks=false,hyperfootnotes=false]{hyperref}
\hypersetup{
			colorlinks=true,
			linkcolor=blue,
			anchorcolor=black,
			citecolor=green,
			urlcolor=blue
			}
 
\usepackage{algorithmicx}
\usepackage[titlenumbered,ruled,vlined]{algorithm2e}
\SetKw{Break}{break}
\usepackage{algpseudocode,float}
\algnewcommand\Input{\textbf{Input: }}
\algnewcommand\Package{\textbf{Package: }}
\algnewcommand\Parameters{\textbf{Parameters: }}
\algnewcommand\Output{\textbf{Output: }}
\algnewcommand{\IfThenElse}[3]{
  \algorithmicif\ #1\ \algorithmicthen\ #2\ \algorithmicelse\ #3}
 \algnewcommand{\IfThen}[2]{
  \algorithmicif\ #1\ \algorithmicthen\ #2}

\usepackage[english]{babel}
\usepackage[backend=bibtex,maxnames=4,url=false,style=numeric,isbn=false, giveninits=true,eprint=false,sorting=nyt]{biblatex}
\bibliography{abbrevs.bib, IP.bib, BL.bib}

\title{Total Variation Regularisation with Spatially Variable Lipschitz Constraints}
\author{Martin Burger\footnote{Department Mathematik, Universit\"{a}t Erlangen-N\"{u}rnberg, Cauerstrasse 11, 91058 Erlangen, Germany. \texttt{martin.burger@fau.de}} \and Yury Korolev\footnote{Department of Applied Mathematics and Theoretical Physics, University of Cambridge, Wilberforce Road, Cambridge CB3 0WA, UK. \texttt{\{yk362;sp751;cbs31\}@cam.ac.uk}} \and Simone Parisotto\footnotemark[2] \and Carola-Bibiane Sch\"onlieb\footnotemark[2]}
\date{}

\newcommand{\RR}{\mathbb{R}}
\newcommand{\xbold}{\bm{x}}

\AtBeginDocument{
  
  \let\argmin\relax

  \let\dist\relax
  \let\div\relax

  \let\prox\relax

  \let\supp\relax

  \DeclareMathOperator*{\argmin}{arg\,min}

  \DeclareMathOperator{\dist}{dist}
  \DeclareMathOperator{\div}{div}

  \DeclareMathOperator{\prox}{\mathbf{prox}}

  \DeclareMathOperator{\supp}{supp}
  
}

\usepackage{caption}
\usepackage[font=scriptsize]{subcaption}

\newcommand{\blank}{\,{\cdot}\,}

\usepackage{tabularx}
\usepackage{booktabs}
\usepackage{multirow}

\usepackage[justification=centering,font=scriptsize]{subcaption}

\usepackage{bm}

\newcommand{\pbold}{{\bm{p}}}
\newcommand{\ybold}{{\bm{y}}}
\newcommand{\zbold}{{\bm{z}}}

\newcommand{\Kcal}{\mathcal{K}}

\newcommand{\diff}{\mathop{}\mathrm{d}}

\usepackage{alphalph}
\allowdisplaybreaks
\begin{document}

\maketitle
\begin{abstract}
    We introduce a first order Total Variation type regulariser that decomposes a function into a part with a given Lipschitz constant (which is also allowed to vary spatially) and a jump part. The kernel of this regulariser contains all functions whose Lipschitz constant does not exceed a given value, hence by locally adjusting this value one can determine how much variation is the reconstruction allowed to have. We prove regularising properties of this functional, study its connections to other Total Variation type regularisers and propose a primal dual optimisation scheme. Our numerical experiments demonstrate that the proposed first order regulariser can achieve reconstruction quality similar to that of second order regularisers such as Total Generalised Variation, while requiring significantly less computational time.
\end{abstract}

\keywords{inverse problems, edge preserving regularisation, total variation, total generalised variation, infimal convolution, primal-dual algorithm}

\ams{65J20, 65J22, 68U10, 94A08}

\section{Introduction}

Edge preserving regularisation plays a crucial role in imaging applications, in particular in image reconstruction~\cite{Chambolle10anintroduction}. Total Variation ($\TV$)~\cite{ROF} is perhaps the most popular edge preserving regularisers since it combines the ability to preserve discontinuities in the reconstructions while allowing for rather efficient computations~\cite{chambolle_pock_acta_numerica:2016}.

A drawback of Total Variation is the so-called \emph{staircasing}~\cite{Ring:2000,Jalalzai:2016}, i.e. the piecewise constant nature of the reconstructions with discontinuities that are not present in the ground truth. To overcome these issues, several regularisers that use second and higher order information (i.e. higher order derivatives) have been introduced. The most successful of them is arguably the Total Generalised Variation ($\TGV$)~\cite{bredies2009tgv}. 

In contrast to Total Variation, which favours piecewise constraint reconstruction, the reconstructions obtained with $\TGV$ are piecewise polynomial; in the most popular case of $\TGV^2$ they are piecewise affine. 

However, $\TGV$ also has some known drawbacks. First, it lacks the \emph{maximum principle}, i.e. the maximum value of the reconstruction can exceed the maximum value of the original function (this statement will be made more precise in Section~\ref{sec:max_pr}). From the numerical point of view, $\TGV$ is typically significantly more expensive than first order methods such as Total Variation.

Therefore, there is an interest in obtaining performance similar to that of $\TGV$ with a first order regulariser, i.e. using only derivatives of the first order. Such approaches use infimal convolution type regularisers~\cite{Burger_TVLp_2016, Burger_TVLp_pt2_2016}, where the Radon norm used in Total Variation is convolved with an $L^p$ norm, $p>1$.

In this work we introduce another infimal convolution type regulariser that is not based on $L^p$ norms, but rather on order intervals in the space of (scalar valued) Radon measures. This allows us to decompose a function into a Lipschitz part and a jump part and to spatially adjust the Lipschitz constant of the Lipschitz part.

We start with the following motivation. Let $\Omega \subset \R^d$ be a bounded Lipschitz domain and $f \in L^2(\Omega)$ a noisy image. Recall the ROF~\cite{ROF} denoising model
\begin{eqnarray}\nonumber
\min_{u \in \BV(\Omega)} \frac{1}{2}\norm{u-f}^2_{L^2(\Omega)}+\alpha \norm{Du}_{\Mg},
\end{eqnarray}
\noindent where $D \colon L^1(\Omega) \to \Mg(\Omega;\R^d)$ is the weak gradient, $\Mg(\Omega;\R^d)$ is the space of vector-valued Radon measures and $\alpha > 0$ is the regularization parameter. 
Introducing an auxiliary variable $g\in\Mg(\Omega;\R^d)$, we can rewrite this problem as follows \looseness=-1
\begin{eqnarray}
\min_{\substack{u \in \BV(\Omega)\\g\in\Mg(\Omega;\R^d)}}~\frac{1}{2}\norm{u-f}^2_{L^2(\Omega)}+\alpha \norm{g}_{\Mg}\qquad s.t.~Du=g.\nonumber
\end{eqnarray}
Our idea is to relax the constraint on $Du$ as follows
\begin{eqnarray}
\min_{\substack{u \in \BV(\Omega)\\g\in\Mg(\Omega;\R^d)}}~\frac{1}{2}\norm{u-f}^2_{L^2(\Omega)}+\alpha \norm{g}_{\Mg}\qquad s.t.~|Du-g|\leq\gamma\nonumber
\end{eqnarray}
\noindent for some positive constant, function or measure $\gamma$. Here $|Du - g|$ is the variation measure corresponding to $Du-g$ and the symbol $"\leq"$ denotes a partial order in the space of signed (scalar valued) measures $\M(\Omega)$. This problem is equivalent to

\begin{eqnarray}\label{eq:tvpwl_inspo}
\min_{\substack{u \in \BV(\Omega)\\g\in\Mg(\Omega;\R^d)}}~\frac{1}{2}\norm{u-f}^2_{L^2(\Omega)}+\alpha \norm{Du-g}_{\Mg}\quad s.t.~|g|\leq\gamma,
\end{eqnarray}
which we take as the starting point of our approach.

The analysis in this paper assumes that the parameter $\gamma \in \M(\omega)$ is given a priori and reflects some knowledge about the solution that we are reconstructing. In our numerical experiments (Section~\ref{sec:num_exp}) we propose a simple procedure for estimating $\gamma$ from the noisy image in the context of denoising, however, this is not the main purpose of the paper. Future work may involve better approaches to estimating $\gamma$ from the data, including learning based approaches.

We also emphasise that the regulariser has the same topolgical properties as Total Variation and hence can be used in general regularisation (and not just denoising) in the same scenarios as Total Variation. 

The paper is organised as follows. In Section~\ref{sec:def} we give three equivalent definitions of the proposed regulariser and study its properties. In Section~\ref{sec:num_impl} we introduce a primal-dual scheme that can be used to solve problem~\eqref{eq:tvpwl_inspo}. Section~\ref{sec:num_exp} contains numerical experiments comparing the performance of $\TV$, $\TGV$ and the proposed regulariser $\TVpwL$.

This paper extends the results of the conference paper~\cite{TVwpL_SSVM}, however, most results presented here are new. The only overlap is Definition~\ref{def:tvpwl_primal} (definition of $\TVpwL$), Theorem~\ref{thm:tvpwl_dual} (dual formulation of $\TVpwL$) and Theorem~\ref{th:tvpwl_topequiv} (topological equivalence to Total Variation). The numerical implementation as a primal-dual scheme and  numerical experiments are also new.

\section{Definition and Properties}\label{sec:def}
In this section we formally define the regularisation functional in~\eqref{eq:tvpwl_inspo}, to which we refer as $\TVpwL^\gamma$. The subscript ``$pwL$'' stands for ``piecewise Lipschitz'' and reflects the fact that, as we shall see, the regulariser promotes reconstructions that are piecewise Lipschitz with (spatially varying) Lipschitz constant $\gamma$. 

  Before we proceed with a formal definition, let us clarify how we understand the inequality sign in~\eqref{eq:tvpwl_inspo}.  Let $\M(\Omega)$ denote the space of all scalar valued finite Radon measures on $\Omega$.
  \begin{definition}\label{def:tvpwl_ordvecsp}
We call a  measure $\mu \in \M(\Omega)$ positive if for every subset $E\subseteq \Omega$ one has $\mu(E)\geq 0$. For two signed measures $\mu_1,\mu_2 \in \M(\Omega)$ we say that $\mu_1\leq\mu_2$ if $\mu_2-\mu_1$ is a positive measure.
\end{definition}
For every $\mu \in \M(\Omega)$, the Hahn decomposition of measures~\cite{DS1} defines two positive measures $\mu_+$ and $\mu_-$ such that
\begin{equation*}
    \mu = \mu_+ - \mu_-
\end{equation*}
and 
\begin{equation*}
    \abs{\mu} = \mu_+ + \mu_-,
\end{equation*}
where $\abs{\mu}$ is the total variation of $\mu$. 


\subsection{Three Equivalent Definitions of $\TVpwL$}
In this section we provide three equivalent definitions of $\TVpwL$. We start with the primal formulation.
\begin{definition}\label{def:tvpwl_primal}
Let $\Omega\subset\R^d$ be a bounded Lipschitz domain, $\gamma\in\M(\Omega)$ be a finite positive measure. For any $u\in L^1(\Omega)$ we define
\begin{eqnarray*}
\TVpwL^\gamma(u):=\min_{g\in\Mg(\Omega;\R^d)}\norm{Du-g}_{\Mg} 
\quad \text{s.t. $|g|\leq\gamma$},
\end{eqnarray*}
where $||\cdot||_{\Mg}$ denotes the Radon norm and $|g|$ is the variation  measure~\cite{Bredies_Lorenz} corresponding to $g$, i.e. for any subset $E \subset \Omega$
\begin{equation*}
|g|(E) \defeq \sup\left\{\sum_{i=1}^\infty \norm{g(E_i)}_2 \mid E = \bigcup_{i \in \mathbb N} E_i, \,\, \text{$E_i$ pairwise disjoint}  \right\}
\end{equation*}
\noindent  (see also the polar decomposition of measures~\cite{Ambrosio}).
\end{definition}

 The use of $\min$ instead of $\inf$ in Definition~\ref{def:tvpwl_primal} is justified, since it is a metric projection onto a closed convex set $\{g \colon \abs{g} \leq \gamma\} \subset \Mg(\Omega;\R^d)$. For $\gamma = 0$, we recover Total Variation, i.e. 
\begin{equation}
    \TVpwL^0 \equiv \TV.
\end{equation}
 
 We can equivalently rewrite Definition~\ref{def:tvpwl_primal} using an infimal convolution
 \begin{equation}
     \TVpwL^\gamma = (\norm{\cdot}_\Mg \square \charf_{\abs{\cdot} \leq \gamma})(Du).
 \end{equation}
 
 It is evident that $\TVpwL$ is lower-semicontinuous and convex.
 
 As with Total Variation, there exists an equivalent dual formulation of $\TVpwL$. The proof of the next result can be found in~\cite{TVwpL_SSVM}, but we include it here for the sake of completeness.
\begin{theorem}\label{thm:tvpwl_dual}
Let $\gamma\in \M(\Omega)$ be a positive finite measure and $\Omega$ a bounded Lipschitz domain. Then for any $u\in L^1(\Omega)$ the $\TVpwL^{\gamma}$ functional can be equivalently expressed as follows
\[
\TVpwL^{\gamma}(u)~=~\sup_{\substack{\phi\in \C_0^{\infty}(\Omega;\R^d) \\ |\phi| \leq 1}}~\left\{ \int_{\Omega} u~\div~\phi~dx-\int_{\Omega} |\phi| d\gamma \right\},
\]
\noindent where $|\phi|$ denotes the pointwise $2$-norm of $\phi$.
\end{theorem}
\begin{proof}
Since by the Riesz-Markov-Kakutani representation theorem the space of vector valued Radon measures $\Mg(\Omega;\R^d)$ is the dual of the space $\C_0(\Omega;\R^d)$, we rewrite the expression in Definition~\ref{def:tvpwl_primal} as follows
\[
\TVpwL^{\gamma}(u)=\inf_{\substack{g\in\Mg(\Omega;\R^d)\\|g|\leq\gamma}}\norm{Du-g}_{\Mg} 
=\inf_{\substack{g\in\Mg(\Omega;\R^d)\\|g|\leq\gamma }}\sup_{\substack{\phi\in\C_0(\Omega;\R^d)\\|\phi| \leq 1}}(Du-g,\phi).
\]

In order to exchange $\inf$ and $\sup$, we need to apply a minimax theorem. In our setting  we can use the Nonsymmetrical Minimax Theorem from \cite[Th. 3.6.4]{Borwein_Zhu}.
Since the set $\{g~|~|g|\leq\gamma\}\subset\Mg(\Omega;\R^d)=(\C_0(\Omega;\R^d))^*$ is bounded,  
convex and closed 
and the set
$\{\phi~|~\norm{\phi}_{2,\infty} \leq 1\}\subset\C_0(\Omega;\R^d)$ is convex, we can swap the infimum and the supremum and obtain the following representation
\[
\begin{aligned}
\TVpwL^{\gamma}(u) &= \sup\limits_{\substack{\phi\in\C_0(\Omega;\R^d)\\|\phi| \leq 1}}\inf\limits_{\substack{g\in\Mg(\Omega;\R^d)\\|g|\leq\gamma }}( Du-g,\phi) \\
 & = \sup\limits_{\substack{\phi\in\C_0(\Omega;\R^d)\\|\phi| \leq 1}}[( Du,\phi)-\sup\limits_{\substack{g\in\Mg(\Omega;\R^d)\\|g|\leq\gamma }}(g,\phi)] 
  = \sup\limits_{\substack{\phi\in\C_0(\Omega;\R^d)\\|\phi| \leq 1}} [(Du,\phi)-(\gamma ,|\phi|)].
\end{aligned}
\]
Noting that the supremum can actually be taken over $\phi\in\C_0^\infty(\Omega;\R^d)$, we obtain
\[
 \TVpwL^{\gamma}(u) = \sup\limits_{\substack{\phi\in\C_0^\infty(\Omega;\R^d)\\|\phi| \leq 1}}[(u, -\div \phi) - (\gamma,|\phi|)]
\]
\noindent which yields the assertion upon replacing $\phi$ with $-\phi$.
\end{proof}
\begin{corollary}\label{cor:jointly_lsc}
It is evident from the dual formulation that $\TVpwL$ is jointly lower-semicontinuous in $u$ and $\gamma$. More precisely, let $u_n \to u$ in $L^1(\Omega)$ and $\gamma_n \weakto^* \gamma$ weakly-$*$ in $\M(\Omega)$, i.e.
\begin{equation*}
    \int_\Omega \phi \, d\gamma_n \to \int_\Omega \phi \, d\gamma
\end{equation*}
for all $\phi \in \C_0(\Omega)$. Then
\begin{eqnarray*}
    \TVpwL^\gamma(u) && = \sup_{\substack{\phi\in \C_0^{\infty}(\Omega;\R^d) \\ |\phi| \leq 1}}~\left\{ \int_{\Omega} u~\div~\phi~dx-\int_{\Omega} |\phi| d\gamma \right\} \\
    && = \sup_{\substack{\phi\in \C_0^{\infty}(\Omega;\R^d) \\ |\phi| \leq 1}} \lim_{n \to \infty} \left\{ \int_{\Omega} u_n~\div~\phi~dx-\int_{\Omega} |\phi| d\gamma_n  \right\} \\
    && \leq \liminf_{n \to \infty} \sup_{\substack{\phi\in \C_0^{\infty}(\Omega;\R^d) \\ |\phi| \leq 1}} \left\{ \int_{\Omega} u_n~\div~\phi~dx-\int_{\Omega} |\phi| d\gamma_n  \right\} \\
    && = \liminf_{n \to \infty} \TVpwL^{\gamma_n}(u_n). 
\end{eqnarray*}
\end{corollary}

The following result provides an alternative definition of $\TVpwL$, which clarifies what kind of features are penalised by $\TVpwL$.

\begin{theorem}\label{thm:tvpwl_third}
Let $\Omega$ be a bounded Lipschitz domain and $\gamma\in\M(\Omega)$ be a finite positive measure. Then for any $u\in L^1(\Omega)$ the functional $TV_{pwL}^{\gamma}$ can be equivalently expressed as follows
\begin{eqnarray*}
\TVpwL^{\gamma}(u)&\ = &\norm{(|Du|-\gamma)_+}_{\M},
\end{eqnarray*}
where $(\cdot)_+$ denotes the positive part of a measure in the sense of Hahn decomposition.
\end{theorem}
\begin{proof}[Proof]
The $\TVpwL$ functional is given by
\begin{equation*}
    \TVpwL^{\gamma}(u) = \min_{\substack{g\in\Mg(\Omega;\R^d) \\ |g|\leq\gamma}} \int_\Omega d\abs{Du-g}.
\end{equation*}
Using the Hahn decomposition~\cite{DS1} we can decompose $\Omega$ into two disjoint subsets where $ |Du|-\gamma \geq 0$ and  $|Du|-\gamma \leq 0$, respectively. Thus, we can split the integral over $\Omega$ as follows
\begin{eqnarray}
\TVpwL^{\gamma}(u)&=&\min_{\substack{g\in\Mg(\Omega;\R^d)\\|g|\leq\gamma }}\int_{|Du|\leq\gamma } d|Du-g| + \int_{|Du|\geq\gamma } d|Du-g|.\nonumber
\end{eqnarray}
Since the two subsets are disjoint, we can optimise over them separately. On $\{u \colon \abs{Du} \leq \gamma\}$ $g =\abs{Du}$ is feasible, hence the first integral vanishes. To estimate the second integral, we observe that for any $A \subset \Omega$
\begin{eqnarray}
|Du|(A) \leq |Du-g|(A) + |g|(A),\nonumber
\end{eqnarray}
hence
\begin{eqnarray*}\label{eq:tvpwl_du-g}
|Du-g|(\{u \colon \abs{Du} \geq \gamma\}) &&\geq |Du| (\{u \colon \abs{Du} \geq \gamma\}) - |g|(\{u \colon \abs{Du} \geq \gamma\}) \\
&& \geq |Du|(\{u \colon \abs{Du} \geq \gamma\}) - \gamma(\{u \colon \abs{Du} \geq \gamma\})
\end{eqnarray*}
and
\begin{eqnarray*}
\TVpwL^{\gamma}(u)&\geq&\int_{|Du|\geq\gamma }(|Du|-\gamma) = \int_{\Omega}(|Du|-\gamma)_+ = \norm{(|Du|-\gamma)_+}_\M.
\end{eqnarray*}
For the converse inequality, consider a sequence $u_n \in C_0^\infty(\Omega)$ such that
\begin{equation*}
    u_n \to u \quad \text{in $L^1(\Omega)$} \quad \text{and} \quad Du_n \to Du \quad \text{in $\Mg(\Omega)$},
\end{equation*}
i.e. $u_n \to u$ in the sense of strict convergence~\cite{Ambrosio}. Consider also a sequence $\gamma_n \in L^1(\Omega)$ such that
\begin{equation*}
    \int_\Omega \phi(x) \, \gamma_n(x) \, dx \to \int_\Omega \phi(x) \, d\gamma(x)
\end{equation*}
for all $\phi \in \C_0(\Omega)$. For every fixed $n$, the minimum is atained if 
\begin{equation*}
    \TVpwL^{\gamma_n}(u_n) = \min_{\substack{g\in L^1(\Omega) \\|g|\leq\gamma_n }} \int_{|Du_n|\geq\gamma_n } \norm{Du_n(x) - g_n(x)}_2 \, dx,
\end{equation*}
where $\norm{\cdot}_2$ denotes the pointwise $2$-norm. For every $x \in \Omega$ we have that
\begin{equation*}
    \norm{Du_n(x) - g_n(x)}_2 = \norm{Du_n(x)}_2 - \norm{g_n(x)}_2,
\end{equation*}
hence
\[
\begin{aligned}
    \TVpwL^{\gamma_n}(u_n) &= \min_{\substack{g\in L^1(\Omega) \\|g|\leq\gamma_n }} \int_{|Du_n|\geq\gamma_n } (\norm{Du_n(x)} - \norm{g_n(x)}_2) \, dx \\
    &= \int_{|Du_n|\geq\gamma_n } (\norm{Du_n(x)} - \gamma_n(x)) \, dx  = \norm{(\abs{Du_n} - \gamma_n)_+}_\M.
\end{aligned}
\]
Since by Corollary~\ref{cor:jointly_lsc} $\TVpwL$ is jointly lower semicontinuous in $u$ and $\gamma$, we get that
\begin{equation*}
    \TVpwL^\gamma(u) \leq \liminf_{n \to \infty} \TVpwL^{\gamma_n}(u_n) = \liminf_{n \to \infty} \norm{(\abs{Du_n} - \gamma_n)_+}_\M = \norm{(\abs{Du} - \gamma)_+}_\M,
\end{equation*}
which proves the assertion.
\end{proof}
\begin{corollary}
It is also clear from the proof that $\TVpwL$ is continuous in $\BV$, i.e. if $u_n \to u$ in $L^1$ and $Du_n \to Du$ in $\Mg(\Omega)$ then $\TVpwL^\gamma(u_n) \to \TVpwL^\gamma(u)$.
\end{corollary}

\subsection{Coercivity}
It is easy to see from Definition~\ref{def:tvpwl_primal} that for any $\gamma \geq 0$ 
\begin{equation*}
    \TVpwL^\gamma(u) \leq \TV(u)
\end{equation*}
for all $u \in L^1$. If $\gamma(\Omega)$ is finite, then we can obtain the converse inequality, up to a constant shift. Therefore, $\TVpwL$ and $\TV$ are topologically equivalent in the sense that one is bounded if and only if the other one is bounded.
\begin{theorem}\label{th:tvpwl_topequiv}
Let $\Omega\subset \R^d$ be a bounded Lipschitz domain and $\gamma\in \M(\Omega)$  a positive finite measure such that $\gamma(\Omega) < \infty$. The for every $u\in L^1(\Omega)$ the following inequalities hold
\[
\TV(u)-\gamma(\Omega)\leq \TVpwL^{\gamma}(u)\leq \TV(u).
\]
\end{theorem}
\begin{proof}
We already established the right inequality. For the left one we observe that for any $g\in\Mg(\Omega;\R^d)$ such that $|g|\leq\gamma$ the following estimate  holds
\begin{equation*}
\norm{Du-g}_{\Mg} \geq \norm{Du}_{\Mg} - \norm{g}_{\Mg} \geq \norm{Du}_{\Mg} - \norm{\gamma}_{\Mg} = \TV(u) - \gamma(\Omega),
\end{equation*}
\noindent which also holds for the infimum over $g$.
\end{proof}

\subsection{Maximum Principle}\label{sec:max_pr}
First order $\TV$-type regularisers typically obey the maximum principle: if $u$ solves the ROF denoising problem
\begin{equation*}
    \min_{u \in U} \frac12 \norm{u-f}^2 + \TV(u),
\end{equation*}
then $\max_x u(x) \leq \max_x f(x)$ and $\min_x u(x) \geq \min_x f(x)$, where  the minima and maxima are understood in the essential sense. Second order regularisers such as Total Generalised Variation ($\TGV$) and second order Total Variation ($\TV^2$)  lack this property. To see this, consider the following simple example.

Let $f \in L^2([-1,1])$ be as follows
\begin{equation*}
    f(x) = \begin{cases}
    -\frac12    \quad &\text{if $-1 \leq x \leq -\frac12$}, \\
    x \quad &\text{if $-\frac12 \leq x \leq \frac12$}, \\
    \frac12 \quad &\text{if $\frac12 \leq x \leq 1$}.
    \end{cases}
\end{equation*}
Consider the following denoising problem using second order Total Variation $\TV^2$~\cite{Chambolle_Lions_TV2:1997}
\begin{equation*}
    \min_u \frac12 \norm{u-f}_2^2 + \alpha\norm{u''}_1,
\end{equation*}
where $u''$ denotes the second derivative of $u$. For a sufficently large regularisation parameter $\alpha$ the solution will lie in the kernel of the regulariser, i.e. it will be affine and by symmetry we can assume that it is linear. Hence, for a sufficiently large $\alpha$, the above problem is equivalent to the following one
\begin{equation*}
    \min_{c \in \R} \int_{-1}^1 (f(x) - cx)^2 \, dx.
\end{equation*}
It is easy to verify that the minimum is attained at $c=\frac{11}{16}$ and $u(1) = \frac{11}{16} > \frac12$. This example is illustrated in Figure~\ref{fig:TV2_max_pr}.

It is known that for some combinations of parameters $\TGV$ reconstructions coincide with those obtained with $\TV^2$~\cite{Scherzer_Poschl_1dTGV:2015,papafitsoros_brdies:2015}, hence the above example also applies to $\TGV$. Even in cases when $\TGV$ produces reconstructions that are different from both $\TV$ and $\TV^2$, the maximum principle can be still violated as examples in~\cite{Scherzer_Poschl_1dTGV:2015} and~\cite{papafitsoros_brdies:2015} demonstrate. For instance, Figure 3 in~\cite{papafitsoros_brdies:2015} shows the results of $\TGV$ denosing of a step function in one dimension and Figure 7.3 in~\cite{Scherzer_Poschl_1dTGV:2015} $\TGV$ denoising of a characteristic function of a subinterval. In both cases we see that the maximum principle is violated.


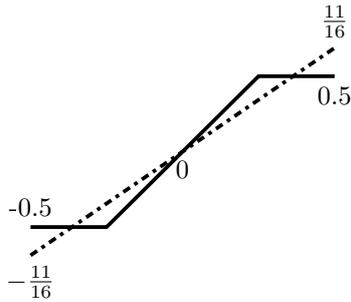
\begin{figure}
    \centering
    \begin{tikzpicture}
    \draw[line width=0.5mm]   (-2*1,-2*0.5) node[align=left,above] {-0.5}  -- (-2*0.5,-2*0.5) -- (2*0.5,2*0.5) -- (2*1,2*0.5) node[align=right,below] {0.5};
    \draw[line width=0.5mm,dash dot] (-2*1,-2*11/16) node[align=left,below] {$-\frac{11}{16}$} -- (2*1,2*11/16) node[align=right,above] {$\frac{11}{16}$};
    (0,0) \node[align=right,below] {0};
    \end{tikzpicture}
    \caption{The $\TV^2$ solution (dash-dotted line) violates the maximum principle by attaining larger ($\frac{11}{16}$) and smaller ($-\frac{11}{16}$) values than the original function (solid line).} 
    \label{fig:TV2_max_pr}
\end{figure}

The following result shows that $\TVpwL$ obeys the maximum principle.

\begin{theorem}
Let $f \in L^2(\Omega)$ and
\begin{equation*}
u = \argmin_{u \in U}  \frac12\norm{u-f}^2 + \TVpwL^\gamma (u).
\end{equation*}
Then $\max_x u(x) \leq \max_x f(x)$ and $\min_x u(x) \geq \min_x f(x)$, where the minima and maxima are understood in the essential sense.
\end{theorem}
\begin{proof}
Denote
\begin{equation*}
    C \defeq \max_x f(x) \quad \text{and} \quad c \defeq \min_x f(x).
\end{equation*} 
and define a cut-off function $\hat u$ as follows
\begin{equation*}
    \hat u \defeq (u \wedge C) \vee c,
\end{equation*}
where $\wedge$ denotes the infimum and $\vee$ the supremum of two functions. Then clearly $\abs{D \hat u} \leq \abs{Du}$ in the sense of measures. Hence
\begin{equation*}
    (\abs{D \hat u} - \gamma)_+ \leq (\abs{Du} - \gamma)_+
\end{equation*}
and using Theorem~\ref{thm:tvpwl_third} we conclude that 
\begin{equation*}
    \TVpwL^\gamma(\hat u) \leq \TVpwL^\gamma(u).
\end{equation*}
It is also clear that $\norm{\hat u - f} < \norm{u - f}$, unless $u=\hat u$. Therefore, 
\begin{equation*}
    \frac12\norm{\hat u-f}^2 + \TVpwL^\gamma (\hat u) < \frac12\norm{u-f}^2 + \TVpwL^\gamma (u).
\end{equation*}
Since $u$ is a minimiser, this is a contradiction and therefore $\hat u =u$. Hence, $c \leq u \leq C$, which proves the assertion.
\end{proof}

\subsection{Characterisation as a Convex Conjugate}

\begin{theorem}\label{Thm:conv_conj}
$\TVpwL$ is the convex conjugate of the following functional $f \colon Z \to \R$, where $Z$ is the pre-dual space of $\BV(\Omega)$~\cite{Burger_Osher_TV_Zoo}
\begin{equation*}
    F(p) \defeq \inf_\phi (\gamma, \abs{\phi}) \quad \text{s.t. $\phi \in \C_0^\infty(\Omega;\R^d)$, $\abs{\phi} \leq 1$ and $D^*\phi = p$}.
\end{equation*}
\end{theorem}
\begin{proof}
First we note that 
\begin{equation*}
    (\gamma, \abs{\phi}) = \sup\limits_{\substack{g\in\Mg(\Omega;\R^d)\\ \abs{g} \leq \gamma}} (\phi, g),
\end{equation*}
hence
\begin{equation*}
    F(p) = \inf\limits_{\substack{\phi \in\C_0^\infty(\Omega;\R^d) \\ \abs{\phi} \leq 1 \\ D^*\phi = p}} \sup\limits_{\substack{g\in\Mg(\Omega;\R^d)\\ \abs{g} \leq \gamma}} (\phi, g).
\end{equation*}
The convex conjugate of $F$ is given by
\begin{eqnarray*}
    F^*(u) &=& \sup_{p \in Z} \left[ (p,u) - \inf\limits_{\substack{\phi \in\C_0^\infty(\Omega;\R^d) \\ \abs{\phi} \leq 1 \\ D^*\phi = p}} \sup\limits_{\substack{g\in\Mg(\Omega;\R^d)\\ \abs{g} \leq \gamma}} (\phi, g) \right] \\
    &=& \sup\limits_{\substack{p \in Z \\ \phi \in\C_0^\infty(\Omega;\R^d) \\ \abs{\phi} \leq 1 \\ D^*\phi = p}} \left[ (p,u) - \sup\limits_{\substack{g\in\Mg(\Omega;\R^d)\\ \abs{g} \leq \gamma}} (\phi, g) \right] \\
    &=& \sup\limits_{\substack{\phi \in\C_0^\infty(\Omega;\R^d) \\ \abs{\phi} \leq 1}} \left[ (D^* \phi,u) - \sup\limits_{\substack{g\in\Mg(\Omega;\R^d)\\ \abs{g} \leq \gamma}} (\phi, g) \right]
\end{eqnarray*}
for any $u \in \BV(\Omega)$.  We further obtain that
\begin{eqnarray*}
    F^*(u) &=& \sup\limits_{\substack{\phi \in\C_0^\infty(\Omega;\R^d) \\ \abs{\phi} \leq 1}} \inf\limits_{\substack{g\in\Mg(\Omega;\R^d)\\ \abs{g} \leq \gamma}} \left[ (\phi,Du) -  (\phi, g) \right].
\end{eqnarray*}
Since $\C_0^\infty(\Omega;\R^d)$ is dense in $\C_0(\Omega;\R^d)$, we can also take the supremum over $\phi \in \C_0(\Omega;\R^d)$ and obtain
\begin{eqnarray*}
    F^*(u) &=& \sup\limits_{\substack{\phi \in\C_0(\Omega;\R^d) \\ \abs{\phi} \leq 1}} \inf\limits_{\substack{g\in\Mg(\Omega;\R^d)\\ \abs{g} \leq \gamma}} \left[ (\phi,Du) -  (\phi, g) \right].
\end{eqnarray*}
Since the set $\{\phi \in\C_0(\Omega;\R^d) \colon \abs{\phi} \leq 1\}$ is convex and the set $\{g\in\Mg(\Omega;\R^d) \colon \abs{g} \leq \gamma\}$ is convex, closed and bounded and $\Mg(\Omega;\R^d) = (\C_0(\Omega;\R^d))^*$, we can apply the Nonsymmetrical Minimax Theorem from \cite[Th. 3.6.4]{Borwein_Zhu} and switch the supremum and maximum, obtaining
\begin{eqnarray*}
    F^*(u) &=& \inf\limits_{\substack{g\in\Mg(\Omega;\R^d)\\ \abs{g} \leq \gamma}} \sup\limits_{\substack{\phi \in\C_0(\Omega;\R^d) \\ \abs{\phi} \leq 1}} \left[ (\phi,Du) -  (\phi, g) \right] \\
    & = & \inf\limits_{\substack{g\in\Mg(\Omega;\R^d)\\ \abs{g} \leq \gamma}} \sup\limits_{\substack{\phi \in\C_0(\Omega;\R^d) \\ \abs{\phi} \leq 1}} \left[ (\phi,Du -g) \right] \\
    &=& \inf\limits_{\substack{g\in\Mg(\Omega;\R^d)\\ \abs{g} \leq \gamma}} \norm{Du-g}_\Mg = \TVpwL^\gamma(u),
\end{eqnarray*}
which proves the assertion.
\end{proof}

\begin{remark}
We notice that for all $p$ the predual of $\TVpwL$ is greater or equal to the predual of $\TV$
\begin{equation*}
    F(p) \geq \charf_{\substack{\phi \in\C_0^\infty(\Omega;\R^d) \\ \abs{\phi} \leq 1 \\ D^*\phi = p}}(p),
\end{equation*}
which agrees with the fact that $\TVpwL(u) \leq \TV(u)$ for all $u$ (convex conjugation is order reversing).
\end{remark}

\subsection{Infimal-Convolution Type Regularisers}
In this section we would like to highlight  connections to infimal convolution type $\TVLp$ regularisers introduced in~\cite{Burger_TVLp_2016,Burger_TVLp_pt2_2016}. For a $u\in L^1(\Omega)$ and $1<p \leq \infty$, $\TVLp(u)$ is defined as follows
\begin{equation}\label{eqn:TVLp}
\TVLp_{\beta}(u) \defeq \inf_{g\in L^p(\Omega;\R^d)}||Du-g||_{\Mg}+\beta||g||_{L^p(\Omega;\R^d)},
\end{equation}
\noindent where $\beta>0$ is a constant. As noted in~\cite{TVwpL_SSVM}, that for a weighted $\infty$-norm, $\TVL^\infty$ and $\TVpwL^\gamma$ are the same thing, provided that the weighting $\beta = \beta(x)$ is chosen appropriately. It turns out that if we optimise jointly over $g \in \Mg(\Omega;\R^d)$ and $\gamma \in L^p(\Omega)$ for $1<p<\infty$, we can recover other $\TVLp$ regularisers.

Consider the following optimisation problem (cf. Definition~\ref{def:tvpwl_primal})
\begin{equation*}
    \inf_{\substack{g\in\Mg(\Omega;\R^d) \\ \abs{g} \leq \gamma \\ \gamma \in L^p(\Omega)}} \norm{Du-g}_{\Mg} + \beta \norm{\gamma}_{L^p(\Omega)}^p.
\end{equation*}
If at an optimal solution $(g^*,\gamma^*)$ the constraint $\abs{g} \leq \gamma$ is inactive in some $\omega \subset \Omega$ with $\abs{\omega}>0$, then we can decrease the value of the objective by choosing $\hat \gamma \defeq \abs{g^*}$. Hence, the constraint $\abs{g} \leq \gamma$ is always active at an optimum and we can write equivalently
\begin{equation*}
    \inf_{g\in\Mg(\Omega;\R^d)} \norm{Du-g}_{\Mg} + \beta \norm{g}_{L^p(\Omega;\R^d)}^p,
\end{equation*}
which is equivalent to~\eqref{eqn:TVLp}.

\section{Numerical Implementation}\label{sec:num_impl}
In this section we will describe a primal-dual scheme we use to solve optimisation problems involving $\TVpwL$. In order to have a fair comparison of different regularisers that is independent of the regularisation parameter, we will solve the following optimisation problem instead of~\eqref{eq:tvpwl_inspo}
\begin{equation}\label{eq:resid_meth}
    \min_u \reg(u) \quad \text{s.t. $\norm{u-f}_2 \leq \delta$},
\end{equation}
where $f$ is the noisy data, $\delta$ is its noise level and $\reg$ is the regulariser; we use $\reg = \TV$; $\TGV$ and $\TVpwL$. Problems~\eqref{eq:resid_meth} and~\eqref{eq:tvpwl_inspo} are equivalent if the regularisation parameter $\alpha$ is chosen according to the discrepancy principle~\cite{engl:1996}.

\subsection{Saddle point problem for $\TVpwL$}
We now provide the details of the numerical implementation of~\eqref{eq:resid_meth} as a saddle-point problem.
From now on we consider our problem in finite dimensions. The Radon norm $\norm{\cdot}_\Mg$ will become $\norm{\cdot}_{2,1}$, where the index $2$ denotes the inner (pointwise) $2$-norm of a vector and $1$ denotes the $1$-norm over the image domain. We will still use the notation $\int_\Omega \diff\xbold$ for the integral over $\Omega$, understanding that it becomes a summation in finite dimensions.

In this section, we will denote the data constraint by $F(u;f) \defeq \charf_{\norm{\cdot-f}_2 \leq \delta} (u)$ and by $R(\ybold)$ the following distance  
\[
    R(\ybold) \defeq \dist_{\mathcal C_\gamma} (\ybold) = \min_{\xbold \in \mathcal C_\gamma} \norm{\ybold-\xbold}_{2,1},
\]
where $\mathcal{C}_\gamma = \{ \ybold :\Omega\to\RR^2 \colon \abs{\ybold(\xbold)} \leq \gamma(\xbold), \,\, \xbold \in \Omega \}$.
Thus, we can rewrite problem \eqref{eq:resid_meth} as follows
\begin{equation}\label{eq:primal}
    \min_u R(\grad u) + F(u;f),
\end{equation}
where $\grad$ denotes the (discrete) gradient.

\begin{lemma}
The Fenchel conjugate of the functional $R$, evaluated at the dual variable $\pbold$, is given by:
\[
R^\ast(\pbold)
= 
\int_\Omega \norm{\pbold(\xbold)}_2 \gamma(\xbold) \diff\xbold
+
\chi_{\{\norm{\blank}_{2,\infty}\leq 1\}}(\pbold).
\]
\end{lemma}
\begin{proof}
We have:
\begin{align*}
R^\ast(\pbold)=\left(\dist_{C_\gamma}(\cdot)\right)^\ast (\pbold) 
&= 
\sup_\ybold \left( \langle \pbold,\ybold \rangle - \dist_{C_\gamma}(\ybold)\right)\\
&= 
\sup_\ybold  \left(\langle \pbold,\ybold \rangle - \inf_{\zbold\in C_\gamma} \norm{\ybold-\zbold}_{2,1} \right)\\
&= 
\sup_{\substack{\ybold \\ \zbold\in C_\gamma}} \left( \langle \pbold,\ybold \rangle  - \norm{\ybold-\zbold}_{2,1}\right)\\
&= 
\sup_{\zbold\in C_\gamma} \left( \sup_\ybold \left( \langle \pbold,\ybold-\zbold \rangle - \norm{\ybold-\zbold}_{2,1} \right) + \langle \pbold,\zbold\rangle\right)\\
&= 
\sup_{\zbold\in C_\gamma} \left( \chi_{\{\norm{\blank}_{2,1}\leq 1\}}(\pbold) + \langle \pbold,\zbold\rangle\right)\\
&= 
\sup_{\zbold\in C_\gamma}  \langle \pbold,\zbold\rangle + \chi_{\{\norm{\blank}_{2,\infty}\leq 1\}}(\pbold)\\
&= 
\sup_{\zbold:\norm{\zbold}_2\leq\gamma}  \langle \pbold,\zbold\rangle + \chi{\{\norm{\blank}_{2,\infty}\leq 1\}}(\pbold)\\
&= 
\sup_{\zbold:\norm{\zbold}_2\leq\gamma}  \int_\Omega \langle \pbold(\xbold), \zbold(\xbold) \rangle \diff\xbold + \chi{\{\norm{\blank}_{2,\infty}\leq 1\}}(\pbold)\\
&=
\int_\Omega\norm{\pbold(\xbold)}_2 \gamma(\xbold)\diff\xbold + \chi{\{\norm{\blank}_{2,\infty}\leq 1\}}(\pbold),
\end{align*}
where the last equality is due to Cauchy-Schwarz
\[
 \langle \pbold(\xbold), \zbold(\xbold)\rangle \leq \norm{\pbold(\xbold)}_2  \norm{\zbold(\xbold)}_2 \leq \norm{\pbold(\xbold)}_2 \gamma(\xbold),
\]
which is also sharp if $\pbold(\xbold)$ and $\zbold(\xbold)$ are parallel.
\end{proof}

Thus, the saddle point problem associated to \eqref{eq:primal} shortened as
\[
\min_u \max_\pbold \langle \grad u,\pbold \rangle - R^\ast(\pbold) + F(u)
\]
reads as follows
\begin{equation}
\min_u \max_\pbold ~ \langle \grad u,\pbold \rangle - \chi_{\{\norm{\blank}_{2,\infty}\leq 1\}}(\pbold) + \int_\Omega \norm{\pbold(\xbold)}_2 \gamma(\xbold)\diff\xbold + F(u).
\label{eq: saddle-point problem}
\end{equation}

The saddle-point optimisation problem \eqref{eq: saddle-point problem} can be solved by using a Primal-Dual Hybrid Gradient (PDHG) scheme from  \cite{chambolle2010first}.
Let $L^2=\norm{\grad}^2$ be the squared operator norm (for which it holds in the discrete setting $L^2\leq 8/h$ when $\grad$ is approximated with a forward finite discretisation scheme on a grid of size $h$, typically $h=1$). 
Recalling that the adjoint of $\grad$ is $\grad^\star=-\div$, then for $\theta\in[0,1]$ and $\tau,\sigma>0$ such that $\tau\sigma L^2< 1$ the PDHG algorithm solving \eqref{eq: saddle-point problem} reads as follows
\begin{equation}
\begin{aligned}
\pbold^{k+1} &= \prox_{\sigma R^\ast}(\pbold^k+\sigma \grad \overline{u}^k),\\
u^{k+1} &= \prox_{\tau F} (u^k-\tau \grad^\ast \pbold^{k+1}),\\
\overline{u}^{k+1} &= u^{k+1} + \theta (u^{k+1}-u^{k}).
\end{aligned}
\label{eq: PDHG}
\end{equation}

In order to apply the scheme described in \eqref{eq: PDHG}, we need explicit  expressions for the proximal mappings $\prox_{\sigma R^\ast}(\blank)$ and $\prox_{\tau F}(\blank)$, which we obtain in Lemmas \ref{lem: proximal of R star} and \ref{lem: proximal of F} below.

\begin{lemma}\label{lem: proximal of R star}
For a given $\pbold^\diamond(\xbold)$, let $\alpha^\ast(\xbold)$ be defined as follows
\begin{equation}
\alpha^\ast(\xbold) = 1-\frac{\sigma \gamma(\xbold)}{\norm{\pbold^\diamond(\xbold)}}_2.
\label{eq: alpha star}
\end{equation}
The proximal map of $R^\ast$ is given by
\begin{equation}
\prox_{\sigma R^\ast}(\pbold^\diamond) =
\begin{dcases}
0 &\text{if } \alpha^\ast(\xbold) \leq 0\text{ i.e.\ } \norm{\pbold^\diamond(\xbold)}_2 \leq \sigma\gamma(\xbold),\\
\norm{\pbold^\diamond(\xbold)}_2^{-1}&\text{if } \alpha^\ast(\xbold) \geq \norm{\pbold^\diamond(\xbold)}_2^{-1}\text{ i.e.\ }\norm{\pbold^\diamond(\xbold)}_2\geq 1 + \sigma \gamma(\xbold),\\
\alpha^\ast(\xbold) &\text{otherwise}.
\end{dcases}
\label{eq: prox R ast}
\end{equation}
\end{lemma}
\begin{proof}
For a given $\pbold^\diamond$, the proximal map of $R^\ast$ is formally written as:
\[
\begin{aligned}
\prox_{\sigma R^\ast}(\pbold^\diamond) 
&= 
\argmin_{\pbold:\norm{\pbold}_{2,\infty}\leq 1} \left(
R^\ast(\pbold) + \frac{1}{2\sigma}\norm{\pbold-\pbold^\diamond}_2^2
\right)
\\
&= 
\argmin_{\pbold:\norm{\pbold}_{2,\infty}\leq 1} 
\left( 
\int_\Omega \norm{\pbold(\xbold)}_2\gamma(\xbold)\diff\xbold
+ 
\frac{1}{2\sigma}\norm{\pbold-\pbold^\diamond}_2^2
\right)
\\
&= 
\argmin_{\pbold:\norm{\pbold}_{2,\infty}\leq 1} 
\left( 
\int_\Omega \left[ \norm{\pbold(\xbold)}_2\gamma(\xbold) + \frac{1}{2\sigma}\norm{\pbold(\xbold)-\pbold^\diamond(\xbold)}^2 \right] \diff\xbold
\right)\\
&= 
\argmin_{\pbold:\norm{\pbold}_{2,\infty}\leq 1} 
\left( 
\int_\Omega \left[ \norm{\pbold(\xbold)}_2\gamma(\xbold) + \frac{1}{2\sigma}\norm{\pbold(\xbold)}_2^2 
- \frac{1}{\sigma} \langle \pbold(\xbold), \pbold^\diamond(\xbold)\rangle + \frac{1}{2\sigma}\norm{\pbold^\diamond(\xbold)}_2^2 \right] \diff\xbold
\right)\\
&= 
\argmin_{\pbold:\norm{\pbold}_{2,\infty}\leq 1} 
\left( 
\int_\Omega \left[ \norm{\pbold(\xbold)}_2\gamma(\xbold) + \frac{1}{2\sigma}\norm{\pbold(\xbold)}_2^2 
- \frac{1}{\sigma} \langle \pbold(\xbold), \pbold^\diamond(\xbold)\rangle \right] \diff\xbold
\right)
\end{aligned}
\]
Since only the term $\langle \pbold(\xbold), \pbold^\diamond(\xbold) \rangle$ depends on the direction of $\pbold(\xbold)$, we can choose $\pbold(\xbold)=\alpha(\xbold)\pbold^\diamond(\xbold)$ with a scalar function $\alpha(\xbold)$ such that
\[
0\leq \alpha(\xbold)\leq \frac{1}{\norm{\pbold^\diamond(\xbold)}_2},
\]
where the second inequality comes from the constraint $\norm{\pbold}_{2,\infty}\leq 1$. Thus we have
\[
\begin{aligned}
\prox_{\sigma R^\ast}(\pbold^\diamond) 
&= 
\argmin_{\alpha(\xbold)\in\left[0,\norm{\pbold^\diamond(\xbold)}_2^{-1}\right]} 
\left( 
\int_\Omega \left[
\alpha(\xbold)\norm{\pbold^\diamond(\xbold)}_2\gamma(\xbold) 
+ \frac{1}{2\sigma}\alpha^2(\xbold)\norm{\pbold^\diamond(\xbold)}_2^2 
- \frac{1}{\sigma}\alpha(\xbold)\norm{\pbold^\diamond(\xbold)}_2^2 \right] \diff\xbold
\right)\\
&= 
\argmin_{\alpha(\xbold)\in\left[0,\norm{\pbold^\diamond(\xbold)}_2^{-1}\right]}
\left( 
\int_\Omega \left[
\alpha(\xbold)\gamma(\xbold) 
+ \frac{1}{2\sigma}\alpha^2(\xbold)\norm{\pbold^\diamond(\xbold)}_2
- \frac{1}{\sigma}\alpha(\xbold)\norm{\pbold^\diamond(\xbold)}_2 \right] \diff\xbold
\right)\\
&= 
\argmin_{\alpha(\xbold)\in\left[0,\norm{\pbold^\diamond(\xbold)}_2^{-1}\right]}
\left( 
\int_\Omega \left[
\alpha(\xbold)
\left( 
\gamma(\xbold)- \frac{1}{\sigma} \norm{\pbold^\diamond(\xbold)}_2
+ 
\frac{1}{2\sigma}\alpha(\xbold)\norm{\pbold^\diamond(\xbold)}_2\right) \right] \diff\xbold
\right),
\end{aligned}
\]
which is a quadratic form with roots
$\alpha_1(\xbold) = 0$, $\alpha_2(\xbold) = 2\left(1-\frac{\sigma \gamma(\xbold)}{\norm{p^\diamond(\xbold)}}_2\right)$ and minimum at
\[
\alpha^\ast(\xbold) = 1-\frac{\sigma \gamma(\xbold)}{\norm{\pbold^\diamond(\xbold)}}_2.
\]
Since  $\alpha(\xbold)$ is constrained to $\left[0,\norm{\pbold^\diamond(\xbold)}_2^{-1}\right]$, the minimum is at zero whenever $\alpha^\ast(\xbold) \leq 0$ and at $\norm{\pbold^\diamond(\xbold)}_2^{-1}$ whenever $\alpha^\ast(\xbold) \geq \norm{\pbold^\diamond(\xbold)}_2^{-1}$. Hence we get that
\[
\prox_{\sigma R^\ast}(\pbold^\diamond) =
\begin{dcases}
0 &\text{if } \alpha^\ast(\xbold) \leq 0\text{ i.e.\ } \norm{\pbold^\diamond(\xbold)}_2 \leq \sigma\gamma(\xbold),\\
\norm{\pbold^\diamond(\xbold)}_2^{-1}&\text{if } \alpha^\ast(\xbold) \geq \norm{\pbold^\diamond(\xbold)}_2^{-1}\text{ i.e.\ }\norm{\pbold^\diamond(\xbold)}_2\geq 1 + \sigma \gamma(\xbold),\\
\alpha^\ast(\xbold) &\text{otherwise}.
\end{dcases}
\]\qedhere
\end{proof}

\begin{lemma}\label{lem: proximal of F}
The proximal map of $F$ for a given $u^\diamond$ is
\begin{equation}
    \prox_{\tau F}(u^\diamond) =
    \begin{dcases}
    u^\diamond &\text{if } \norm{u^\diamond-f}_2\leq \delta,\\
    f+\frac{u^\diamond-f}{\norm{u^\diamond-f}_2}\delta &\text{if } \norm{u^\diamond-f}_2 > \delta.
    \end{dcases}
\label{eq: prox F}
\end{equation}
\end{lemma}
\begin{proof}
The proof is straightforward and it is based on a simple projection onto the constraint:
\[
\begin{aligned}
    \prox_{\tau F}(u^\diamond) 
    &= 
    \argmin_u \chi_{\{\norm{\blank-f}_2\leq \delta\}}(u) + \frac{1}{2\tau}\norm{u-u^\diamond}_2^2\\
    &= 
    \argmin_{u:\norm{u-f}_2\leq \delta} \frac{1}{2\tau}\norm{u-u^\diamond}_2^2=
    \begin{dcases}
    u^\diamond &\text{if } \norm{u^\diamond-f}_2\leq \delta,\\
    f+\frac{u^\diamond-f}{\norm{u^\diamond-f}_2}\delta &\text{if } \norm{u^\diamond-f}_2 > \delta.
    \end{dcases}
\end{aligned}
\]
\end{proof}

As a stopping criterion for the iterations in \eqref{eq: PDHG}, we compute the difference between two iterates of our Primal-Dual Algorithm as it is done in~\cite{goldstein2015adaptive}: 
\begin{equation}
\texttt{residual} := \frac{1}{M\cdot N}\left(\left|\frac{u^k - u^{k+1} - \tau\left(\grad^\ast (\pbold^k- \pbold^{k+1})\right)}{\tau}\right| + \left|\frac{\pbold^k - \pbold^{k+1} - \sigma \left(\grad (u^k - \overline{u}^{k+1})\right)}{\sigma}\right|\right).
\label{eq: residual}
\end{equation}

Now that the proximal maps of $R^\ast$ and $F$ are available, we have all the ingredients for the Primal-Dual Hybrid Gradient (PDHG) scheme in \ref{eq: PDHG}, detailed in Algorithm \ref{alg: PDHG}. The source code is available online\footnote{The MATLAB code is freely available at \url{https://github.com/simoneparisotto}}.

\begin{algorithm*}[tb]\noindent 
\SetAlgoLined\footnotesize
\caption{PDHG: Primal Dual Hybrid Gradient scheme for solving \eqref{eq: saddle-point problem}}
\label{alg: PDHG}
\SetKwData{maxiter}{maxiter}
\SetKwData{tol}{tol}
\SetKwData{residual}{residual}
\SetKwData{prox}{prox}
\SetKwData{colourchannel}{colour channel}
\SetKwInOut{Input}{Input}
\SetKwInOut{Output}{Output}
\SetKwInOut{Parameters}{Parameters}
\SetKwProg{Fn}{Function}{:}{}
\SetKwFunction{PDHG}{PDHG\_TVpw}
\Input{a noisy image $f$ of size $M\times N$, an estimation $\gamma$ of $\norm{\grad u}_2$, a bound $\delta>0$ for the $L^2$-norm;}
\Output{the denoised image $u$;}
\Parameters{maximum number of iterations (\maxiter), exit tolerance for the residual (\tol), $\sigma,\tau>0$ such that $\sigma\tau L^2<1$, with $L^2=8$.}
\BlankLine
\Fn{\PDHG}{
\BlankLine
$u^0=\overline{u}^0=f, \pbold^0=\Kcal u^0$\tcp*{Initialisation}
\BlankLine
\For{$k=1,\dots,\maxiter$}
{
\BlankLine
\tcp{Dual problem}
$\pbold^\diamond = \pbold^k + \sigma \Kcal \overline{u}^k$\;
$\pbold^{k+1} = 
\begin{dcases}
0 &\text{if } \norm{\pbold^\diamond}_2 \leq \sigma\gamma,\\
\norm{\pbold^\diamond}_2^{-1}&\text{if } \norm{\pbold^\diamond}_2\geq 1 + \sigma \gamma\\
1-\sigma\gamma \norm{\pbold^\diamond}_2^{-1} &\text{otherwise}.
\end{dcases}
$\tcp*{$\prox_{\sigma R^\ast}(\pbold^\diamond)$, see \eqref{eq: prox R ast}}
\BlankLine
\tcp{Primal problem}
$u^\diamond = u^k-\tau\Kcal\pbold^{k+1}$\;
$u^{k+1} =     
\begin{dcases}
u^\diamond &\text{if } \norm{u^\diamond-f}_2\leq \delta,\\
f+\frac{u^\diamond-f}{\norm{u^\diamond-f}_2}\delta &\text{if } \norm{u^\diamond-f}_2 > \delta.
\end{dcases}
$
\tcp*{$\prox_{\tau F}$, see \eqref{eq: prox F}}
\BlankLine
\tcp{Extrapolation}
$\overline{u}^{k+1} = u^{k+1} + \theta (u^{k+1}-u^k)$\;
\BlankLine
\tcp{Computation of the residual and exit condition}
$\residual = \frac{1}{M\cdot N}\left(\left|\frac{u^k - u^{k+1} - \tau\left(\Kcal^\ast (\pbold^k- \pbold^{k+1})\right)}{\tau}\right| + \left|\frac{\pbold^k - \pbold^{k+1} - \sigma \left(\Kcal (u^k - \overline{u}^{k+1})\right)}{\sigma}\right|\right)$ \;
\IfThen{
$\residual \leq \tol$
}{
\Break}\;
}
$u^\ast = u^{k+1}$\;
}
\Return
\end{algorithm*}

\paragraph{Discretisation.}
In the discrete setting, $\Omega$ is an imaging domain, i.e.\ a rectangular grid of $M\times N$ pixels, while $u$ denotes the grey-scale image of height $M$ and width $N$ pixels, defined over $\Omega$ and taking values in the intensity range $[0,255]$. 
The scalar value $u_{i,j}$ is associated with the intensity value of the image in the position $(i,j)$ of the imaging domain.
To generate the differential operator $\grad$ and its adjoint $\grad^\ast=-\div$, we use the forward finite difference scheme with Neumann boundary conditions. In particular, $(\grad u)_{i,j}=(\partial_1 u, \partial_2 u)_{i,j}$ reads as follows
\[
(\partial_1 u)_{i,j} = 
\begin{dcases}
\frac{u_{i+1,j}-u_{i,j}}{h} &\text{if } i<M,\\
0 &\text{if } i=M,
\end{dcases}
\quad
\text{and}
\quad
(\partial_2 u)_{i,j} = 
\begin{dcases}
\frac{u_{i,j+1}-u_{i,j}}{h} &\text{if } j<N,\\
0 &\text{if } j=N.
\end{dcases}
\]
The divergence $(\div \pbold)_{i,j}$ is defined for the auxiliary variable $\pbold=(p_1,p_2)$ as follows:
\[
(\div \pbold)_{i,j}
=
\begin{dcases}
\frac{(p_1)_{i,j}}{h} &\text{if } i=1,\\
\frac{(p_1)_{i,j}-(p_1)_{i-1,j}}{h} &\text{if } i=(1,M),\\
-\frac{(p_1)_{i-1,j}}{h} &\text{if } i=M,
\end{dcases}
+
\begin{dcases}
\frac{(p_2)_{i,j}}{h} &\text{if } j=1,\\
\frac{(p_2)_{i,j}-(p_2)_{i,j-1}}{h} &\text{if } j=(1,N),\\
\frac{(p_2)_{i,j-1}}{h} &\text{if } j=N.
\end{dcases}
\]


\section{Numerical results}\label{sec:num_exp}
In this section, we compare the performance of three regulraisers: $\TV$, $\TGV$ and $\TVpwL$ in problem~\eqref{eq:resid_meth}. We use the primal dual scheme introduced earlier as well as, for the sake of comparison, CVX (a package for specifying and solving convex programs \cite{cvx,gb08}, used here with \texttt{default} precision). To generate the differential operator for the use in CVX, we use the DIFFOP package~\cite{diffop}.


\paragraph{Dataset.}
Our dataset is composed of several natural grey-scale images of size $256\times 256$ pixels displayed in Figure \ref{fig: dataset}. The images are taken from ImageNet (\url{http://http://www.image-net.org/}) and from \url{http://decsai.ugr.es/cvg/dbimagenes/g512.php} and are free to use. In our experiments we add $10\%$ and $20\%$  additive Gaussian noise to our images, i.e.\ the noisy data $f$ is given by
\[
f(\xbold) = u_{\texttt{GT}}(\xbold) + n(\xbold),
\]
where $u_{\texttt{GT}}$ is the ground truth image, $n(\xbold)$ is Gaussian  noise of zero mean and variance 0.1*255 or 0.2*255 for $10\%$ and $20\%$, respectively. The intensity range of ground truth images is $[0,255]$.

\begin{figure}[tbhp]
\centering
\begin{subfigure}[t]{0.15\textwidth}\centering
\includegraphics[width=1\textwidth]{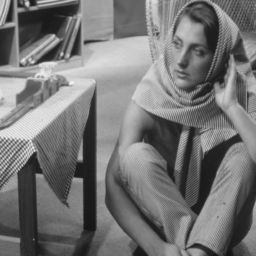}
\caption{\texttt{barbara}}
\end{subfigure}
\begin{subfigure}[t]{0.15\textwidth}\centering
\includegraphics[width=1\textwidth]{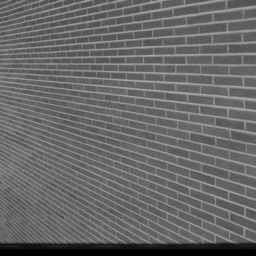}
\caption{\texttt{brickwall}}
\end{subfigure}
\begin{subfigure}[t]{0.15\textwidth}\centering
\includegraphics[width=1\textwidth]{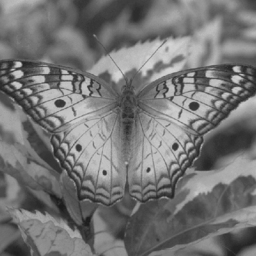}
\caption{\texttt{butterfly}}
\end{subfigure}
\begin{subfigure}[t]{0.15\textwidth}\centering
\includegraphics[width=1\textwidth]{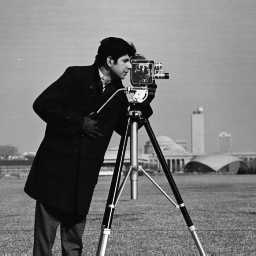}
\caption{\texttt{cameraman}}
\end{subfigure}
\begin{subfigure}[t]{0.15\textwidth}\centering
\includegraphics[width=1\textwidth]{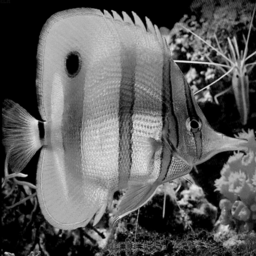}
\caption{\texttt{fish}}
\end{subfigure}
\begin{subfigure}[t]{0.15\textwidth}\centering
\includegraphics[width=1\textwidth]{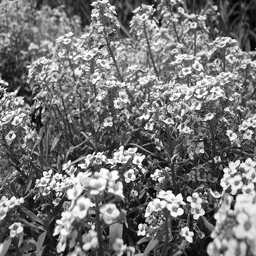}
\caption{\texttt{flowers}}
\end{subfigure}
\\
\begin{subfigure}[t]{0.15\textwidth}\centering
\includegraphics[width=1\textwidth]{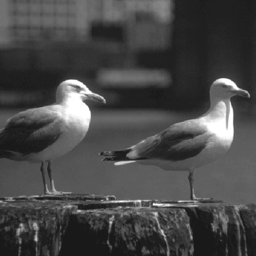}
\caption{\texttt{gull}}
\end{subfigure}
\begin{subfigure}[t]{0.15\textwidth}\centering
\includegraphics[width=1\textwidth]{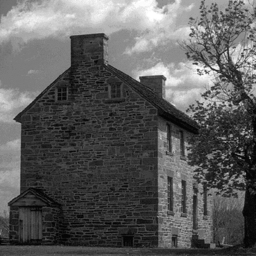}
\caption{\texttt{house}}
\end{subfigure}
\begin{subfigure}[t]{0.15\textwidth}\centering
\includegraphics[width=1\textwidth]{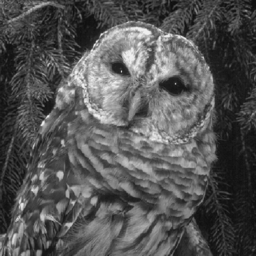}
\caption{\texttt{owl}}
\end{subfigure}
\begin{subfigure}[t]{0.15\textwidth}\centering
\includegraphics[width=1\textwidth]{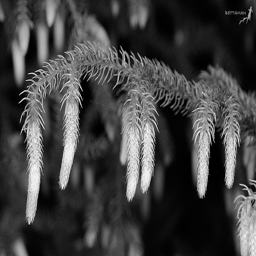}
\caption{\texttt{pine\_tree}}
\end{subfigure}
\begin{subfigure}[t]{0.15\textwidth}\centering
\includegraphics[width=1\textwidth]{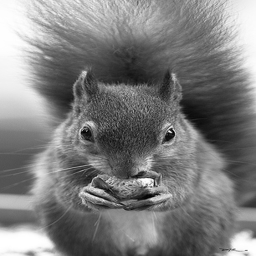}
\caption{\texttt{squirrel}}
\end{subfigure}
\begin{subfigure}[t]{0.15\textwidth}\centering
\includegraphics[width=1\textwidth]{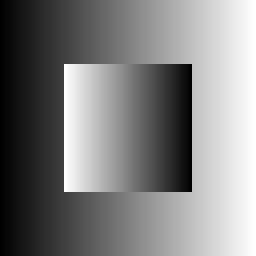}
\caption{\texttt{synthetic}}
\end{subfigure}
\caption{Our dataset of images with size $256\times 256$ pixels. Images are free to use. Images (f), (j) and (k) are from ImageNet (\url{http://http://www.image-net.org/}), other images can be downloaded from \url{http://decsai.ugr.es/cvg/dbimagenes/g512.php}.}
\label{fig: dataset}
\end{figure}

\paragraph{Parameter choice.}
The study of strategies of estimating the parameter $\gamma$ of $\TVpwL^\gamma$ is beyond the scope of our paper, which assumes that a good estimate of $\gamma$ has been already obtained. We will use the simple pipeline of estimating $\gamma$ based on overregularised $\TV$ reconstructions presented in~\cite{TVwpL_SSVM} without claiming its optimality. These reconstructions will be referred to as ``over-$\TV$''. To demonstrate the best possible performance of $\TVpwL^\gamma$ in the idealistic scenario of exact $\gamma$, we also estimate $\gamma$ using the magnitude (but not the direction) of the gradient of the ground truth image. These reconstructions will be referred to as ``GT''.

The pipeline from~\cite{TVwpL_SSVM} can be summarised as follows. We first denoise $f$ using the ROF model 
\begin{equation}
\widehat{u}=\argmin_u \lambda \TV(u) + \frac{1}{2}\norm{u-f}_2^2.
\label{eq: over-TV argmin}
\end{equation}
with a large parameter $\lambda>0$. We choose $\lambda=500$ and solve~\eqref{eq: over-TV argmin} with a standard Primal-Dual algorithm~\cite{chambolle2010first}. 
Once $\widehat{u}$ is available, we compute the residual $r \defeq f-\widehat{u}$ and smooth it with a Gaussian filter with kernel $K_\rho$ of standard deviation $\rho>0$, in our experiments $\rho=2$, to obtain $r_\rho \defeq K_\rho \ast r$. The parameter $\gamma$ is estimated from the filtered residual as $\gamma = \abs{\grad r_\rho}$, where $\abs{\cdot}$ denotes the pointwise $2$-norm.

For the $\TGV$ denoising problem
\begin{equation*}
    \min_{\substack{u \in \BV(\Omega) \\ w \in \BD (\Omega)}} \norm{Du-w}_\M + \beta \norm{w}_\M \quad \text{s.t. $\norm{u-f}_2 \leq \delta$},
\end{equation*}
where $\BD(\Omega)$ is the space of vector fields
of bounded deformation on $\Omega$~\cite{bredies2009tgv}, we choose $\beta = 1.25$, which is in the optimal range $[1,1.5]$ reported in~\cite{DlosR_CS_TV_bilevel:2018}.

\paragraph{A synthetic image.}
As a toy example, in Figure \ref{fig: synthetic} we show the results for a synthetic image corrupted with $10\%$ of Gaussian noise. This is image is piecewise-affine, making it ideal for second order $\TGV$. The results for $\TV$ and $\TGV$ are shown in Figures~\ref{fig: synthetic TV} and~\ref{fig: synthetic TGV}, respectively. In Figures~\ref{fig: synthetic over-tv} -- \ref{fig: synthetic TVpwL_over_TV} we show the pipeline for estimating $\gamma$ as described above and the final result obtained using $\TVpwL$ with this $\gamma$. We notice the that staircasing in the $\TVpwL$ reconstruction (Figure~\ref{fig: synthetic TVpwL_over_TV}) is significantly reduced compared to the $\TV$ reconstruction (Figure~\ref{fig: synthetic TV}). In fact, the $\TVpwL$ reconstruction is rather close to the one obtained using $\TGV$ (Figure~\ref{fig: synthetic TGV}). If we compare the cpu time needed to compute these reconstructions, we notice that $\TGV$ is about $5$ times slower. In numerical experiments with natural images (that will follow) we will see that $\TVpwL$ can be an order of magnitude faster than $\TGV$.

In order to show the best performance that $\TVpwL$ could obtain with the best possible information about the norm of the gradient, in Figure~\ref{fig: synthetic TVpwL_GT} we demonstrate the results obtained using $\gamma$ estimated from the ground truth.

\paragraph{Convergence.}
In Figure \ref{fig: PD gap} we report the ``Primal-Dual residual'' (\texttt{residual}) defined in \ref{eq: residual} for the case of the \texttt{synthetic} results in Figure \ref{fig: synthetic}.
We observe that for all regularisers the decay of the $\texttt{residual}$ (in red) is sub-linear when the fidelity constraint is far from being an equality, i.e.\ $\frac{\delta}{\norm{u-f}_2} \gg 1$; once this constraint gets close to being an equality, i.e.\ $\frac{\delta}{\norm{u-f}_2} \to 1^+$, the decay turns out to have the expected second-order behaviour. 

\begin{figure}[tbh]
    \centering
    \begin{subfigure}[t]{0.245\textwidth}
    \centering
    \includegraphics[width=0.605\textwidth]{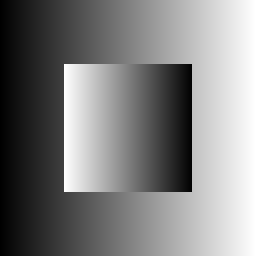}
    \caption{Original (GT)}
    \end{subfigure}
    \begin{subfigure}[t]{0.245\textwidth}
    \centering
    \includegraphics[width=0.605\textwidth]{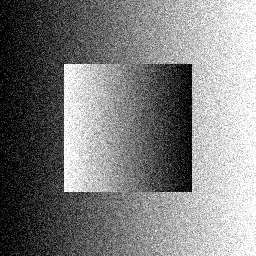}
    \caption{Noisy\\($10\%$ Gauss.\ noise)}
    \label{fig: synthetic noise}
    \end{subfigure}
    \begin{subfigure}[t]{0.245\textwidth}
    \centering
    \includegraphics[width=0.605\textwidth]{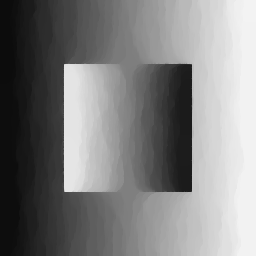}
    \caption{$\TV$\\SSIM: 0.945, PSNR: 33.44\\cputime: 14.33 s.}
    \label{fig: synthetic TV}
    \end{subfigure}
    \begin{subfigure}[t]{0.245\textwidth}
    \centering
    \includegraphics[width=0.605\textwidth]{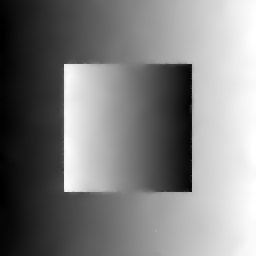}
    \caption{$\TGV^2$\\SSIM: 0.987, PSNR: 37.99\\cputime: 115.89}
    \label{fig: synthetic TGV}
    \end{subfigure}
    \\
    \begin{subfigure}[t]{0.245\textwidth}
    \centering
    \includegraphics[width=0.605\textwidth]{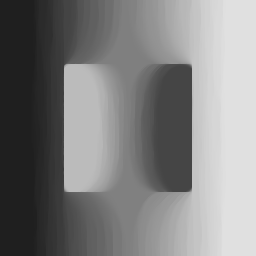}
    \caption{$\widehat{u}$ (over-TV)}
    \label{fig: synthetic over-tv}
    \end{subfigure}
    \begin{subfigure}[t]{0.245\textwidth}
    \centering
    \includegraphics[width=0.605\textwidth]{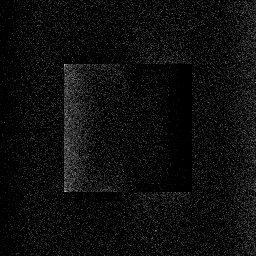}
    \caption{Residual $r$\\(from over-TV)}
    \end{subfigure}
    \begin{subfigure}[t]{0.245\textwidth}
    \centering
    \includegraphics[width=0.605\textwidth]{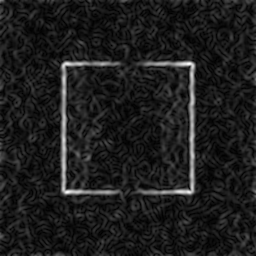}
    \caption{$\gamma$ (rescaled)\\(from over-TV)}
    \end{subfigure}
    \begin{subfigure}[t]{0.245\textwidth}
    \centering
    \includegraphics[width=0.605\textwidth]{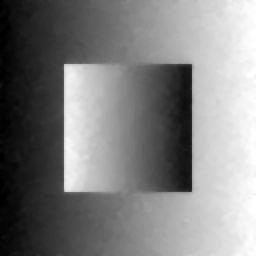}
    \caption{$\TVpwL$ (over-TV)\\SSIM: 0.953, PSNR: 32.63\\cputime: 24.19 s.}
    \label{fig: synthetic TVpwL_over_TV}
    \end{subfigure}
    \\
    \begin{subfigure}[t]{0.245\textwidth}
    \hfill
    \end{subfigure}
    \begin{subfigure}[t]{0.245\textwidth}
    \centering
    \includegraphics[width=0.605\textwidth]{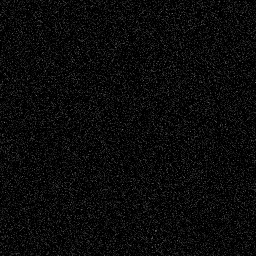}
    \caption{Residual $r$\\(from GT)}
    \end{subfigure}
    \begin{subfigure}[t]{0.245\textwidth}
    \centering
    \includegraphics[width=0.605\textwidth]{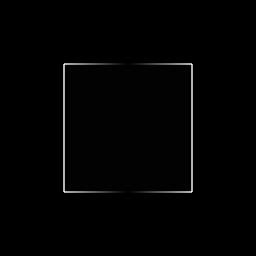}
    \caption{$\gamma$ (rescaled)\\(from GT)}
    \end{subfigure}
    \begin{subfigure}[t]{0.245\textwidth}
    \centering
    \includegraphics[width=0.605\textwidth]{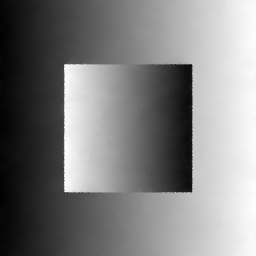}
    \caption{$\TVpwL$ (GT)\\SSIM: 0.980, PSNR: 34.11\\cputime: 14.13 s.}
    \label{fig: synthetic TVpwL_GT}
    \end{subfigure}
    \caption{The \texttt{synthetic} image. The full denoising workflow of Figure \ref{fig: synthetic noise} is displayed: in the second row with $\gamma$ computed using an overregularised $\TV$ reconstruction and in the trid row using the ground-truth $\gamma$.}
    \label{fig: synthetic}
    \begin{subfigure}[t]{0.245\textwidth}
    \centering
    \includegraphics[width=1\textwidth]{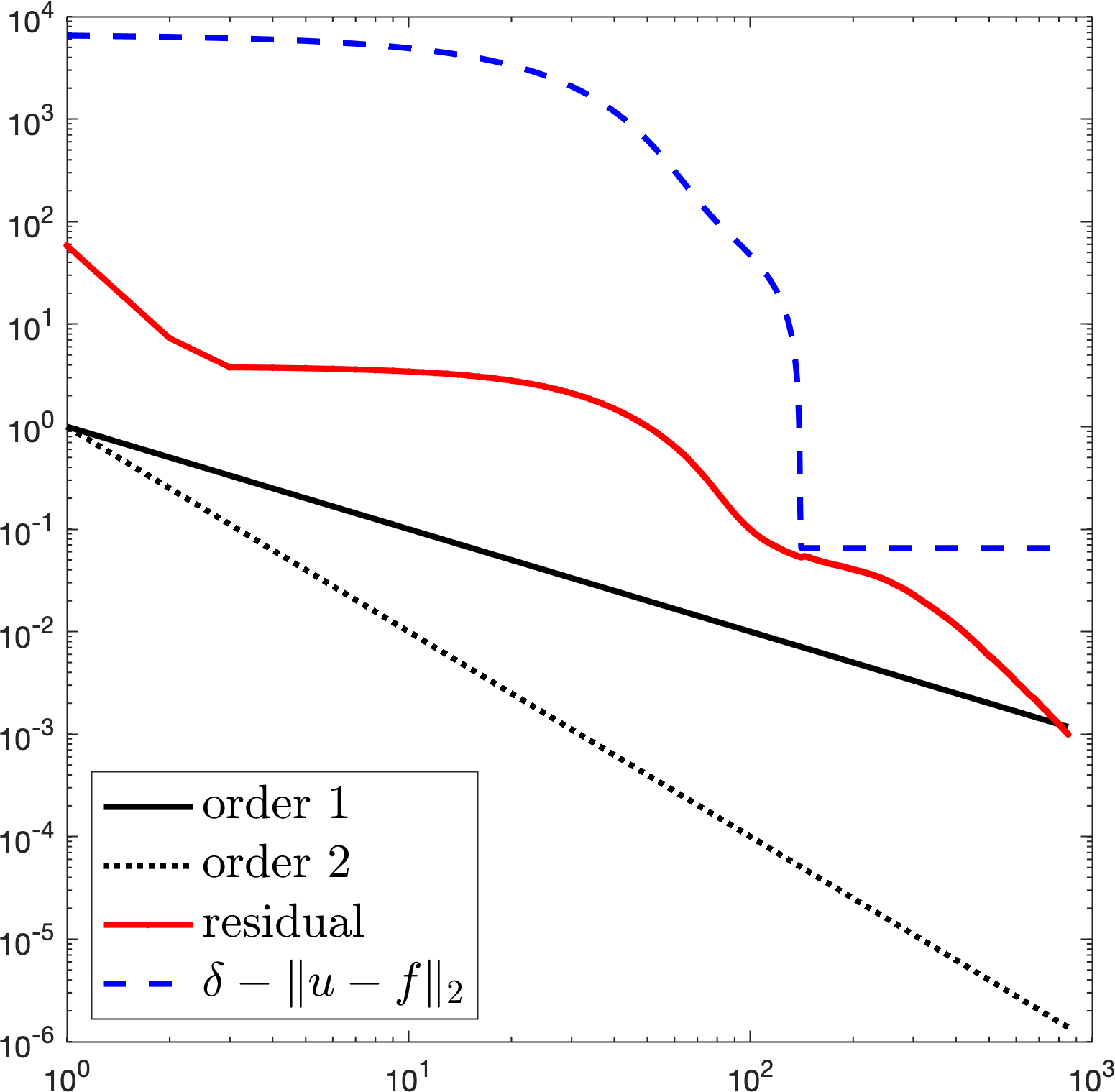}
    \caption{$\TV$}
    \end{subfigure}
    \begin{subfigure}[t]{0.245\textwidth}
    \centering
    \includegraphics[width=1\textwidth]{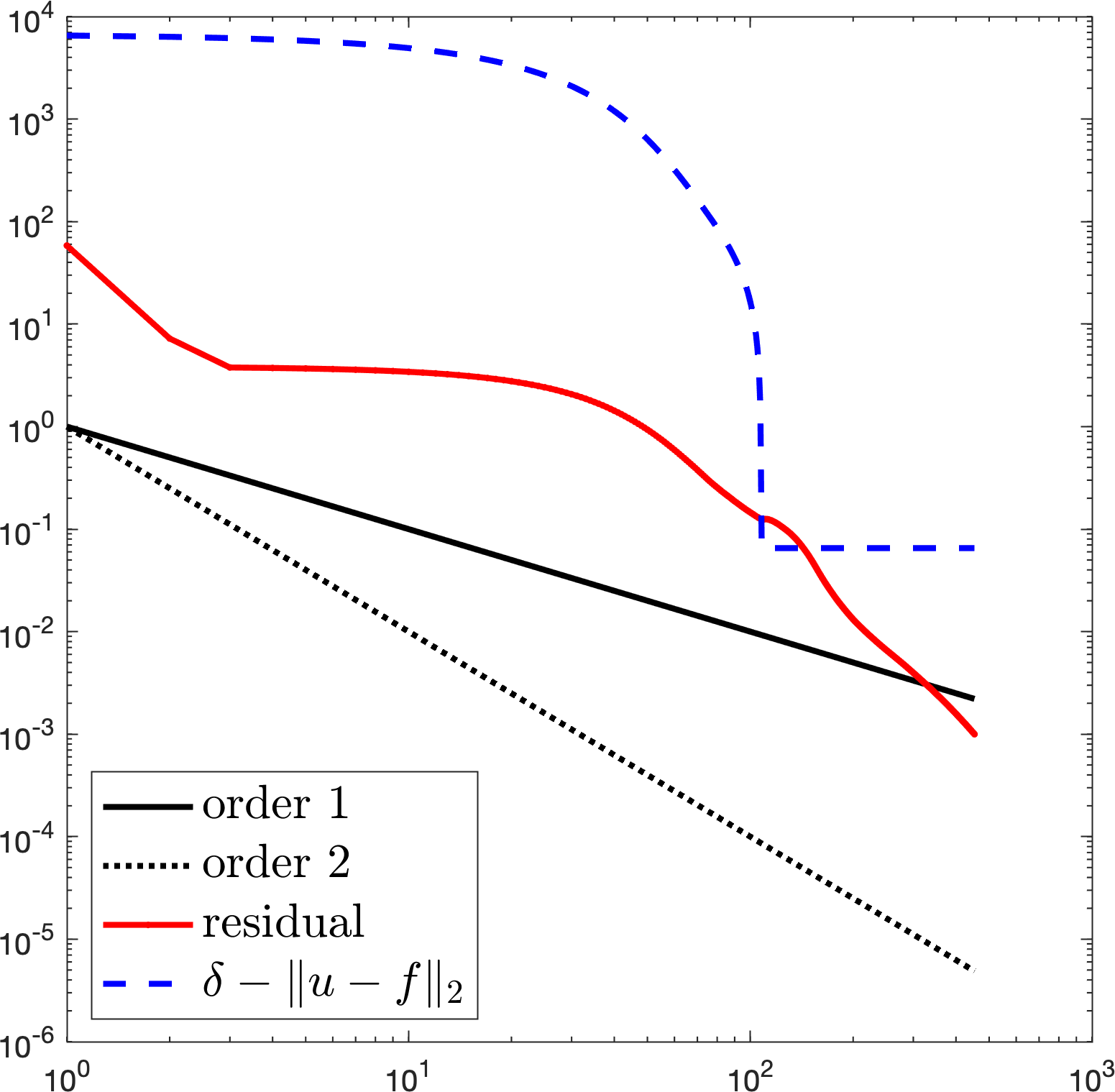}
    \caption{$\TVpwL$ (GT)}
    \end{subfigure}
    \begin{subfigure}[t]{0.245\textwidth}
    \centering
    \includegraphics[width=1\textwidth]{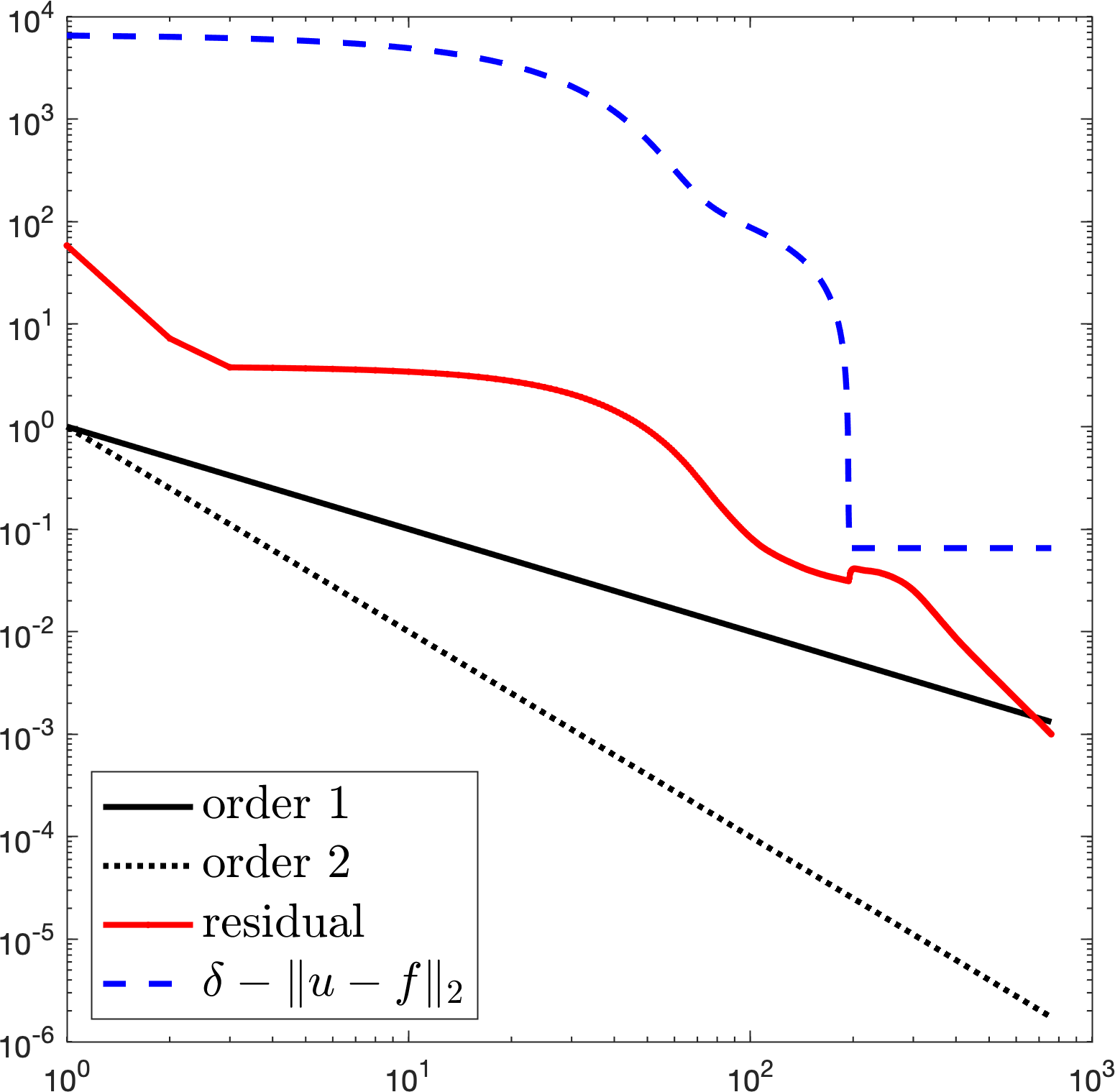}
    \caption{$\TVpwL$ (over-TV)}
    \end{subfigure}
    \begin{subfigure}[t]{0.245\textwidth}
    \centering
    \includegraphics[width=1\textwidth]{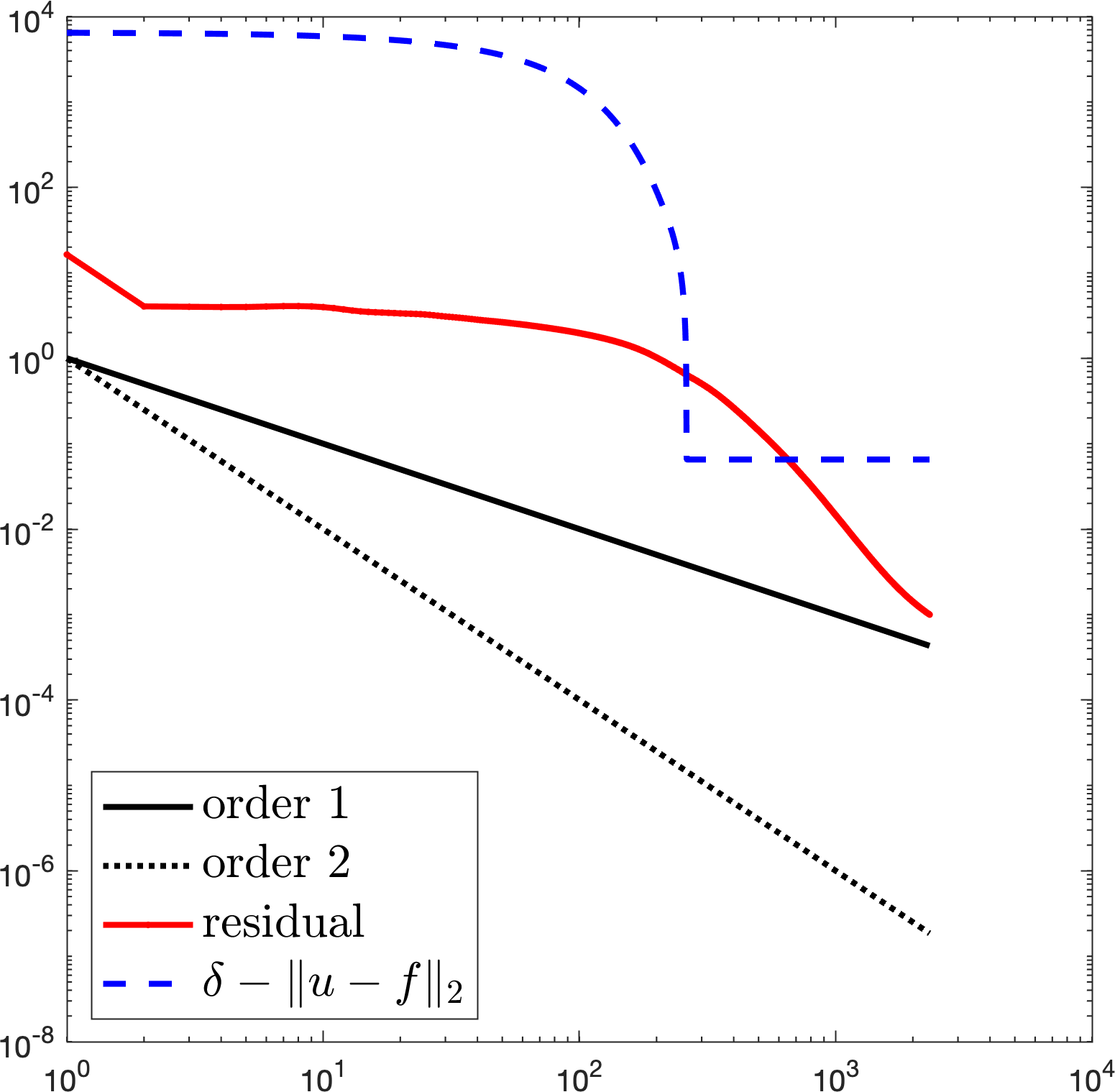}
    \caption{$\TGV^2$}
    \end{subfigure}
    \caption{Loglog plot decay of the residual (in red) and gap constraint $\delta-\norm{u-f}_2$ (in dashed blue) for the \texttt{synthetic} image in Figure \ref{fig: synthetic} (corrupted with 10\% Gaussian noise); in continuous black order 1 and dotted black order 2 of decay. The exit tolerance for the residual is set to $\mathtt{tol}=\texttt{1e-03}$.}
    \label{fig: PD gap}
\end{figure}

\paragraph{Real images.}

In this section, we compare the performances of the PDHG Algorithm \ref{alg: PDHG} (and exit condition $\texttt{tol}=\texttt{1e-3}$ in the residual) with respect to CVX.  
All our experiments are carried out in MATLAB 2019a, on a MacBook Pro 2019 (2.4 GHz Intel Core i5, RAM 16 GB 2133 MHz LPDDR3). 
Quantitative results (the values of $\SSIM$, $\PSNR$ and cpu time) are reported in Table \ref{tab: quantitative results 10 Gaussian}. 

In Figure \ref{fig: computation of gamma} we report the estimation of $\gamma$ using either the over-regularised $\TV$ reconstruction or the ground truth for a selection of real images from our dataset in Figure \ref{fig: dataset} and for different noise levels ($10\%$ vs.\ $20\%$).

\begin{figure}[tbh]
    \centering
\begin{subfigure}[t]{0.15\textwidth}
    \centering
    \includegraphics[width=1\textwidth]{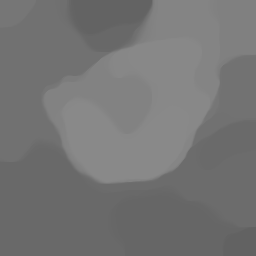}
    \caption{$\widehat{u}$ (over-TV)\\ from $10\%$ noise}
    \label{fig: TV_buttefly_10}
    \end{subfigure}
    \begin{subfigure}[t]{0.15\textwidth}
    \centering
    \includegraphics[width=1\textwidth]{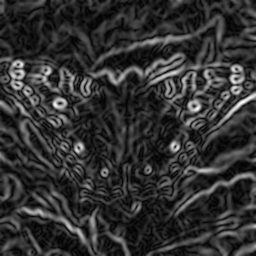}
    \caption{$\gamma$ (over-TV)\\ from $10\%$ noise}
    \label{fig: gamma_buttefly_10}
    \end{subfigure}
    \begin{subfigure}[t]{0.15\textwidth}
    \centering
    \includegraphics[width=1\textwidth]{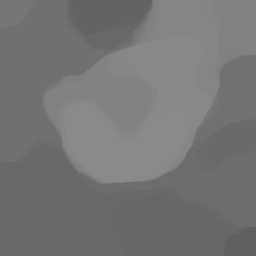}
    \caption{$\widehat{u}$ (over-TV)\\from 20\% noise}
    \label{fig: TV_buttefly_20}
    \end{subfigure}
    \begin{subfigure}[t]{0.15\textwidth}
    \centering
    \includegraphics[width=1\textwidth]{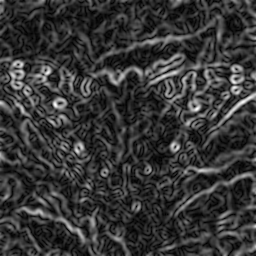}
    \caption{$\gamma$ (over-TV)\\from 20\% noise}
    \label{fig: gamma_buttefly_20}
    \end{subfigure}
    \begin{subfigure}[t]{0.15\textwidth}
    \centering
    \includegraphics[width=1\textwidth]{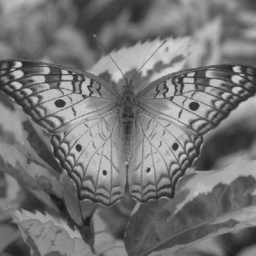}
    \caption{$\widehat{u}$ (GT)}
    \label{fig: butterfly_GT}
    \end{subfigure}
    \begin{subfigure}[t]{0.15\textwidth}
    \centering
    \includegraphics[width=1\textwidth]{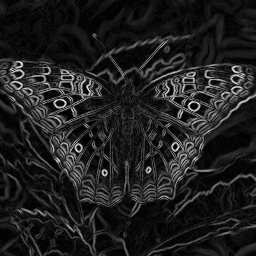}
    \caption{$\gamma$ from GT}
    \label{fig: gamma_buttefly_GT}
    \end{subfigure}
    \\    
    \begin{subfigure}[t]{0.15\textwidth}
    \centering
    \includegraphics[width=1\textwidth]{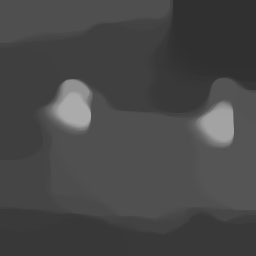}
    \caption{$\widehat{u}$ (over-TV)\\from 10\% noise}
    \label{fig: TV_gull_10}
    \end{subfigure}
    \begin{subfigure}[t]{0.15\textwidth}
    \centering
    \includegraphics[width=1\textwidth]{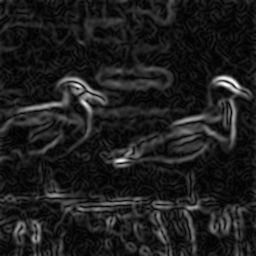}
    \caption{$\gamma$ (over-TV)\\from 10\% noise}
    \label{fig: gamma_gull_10}
    \end{subfigure}
    \begin{subfigure}[t]{0.15\textwidth}
    \centering
    \includegraphics[width=1\textwidth]{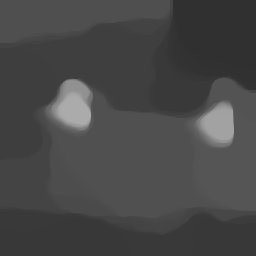}
    \caption{$\widehat{u}$ (over-TV)\\from 20\% noise}
    \label{fig: TV_gull_20}
    \end{subfigure}
    \begin{subfigure}[t]{0.15\textwidth}
    \centering
    \includegraphics[width=1\textwidth]{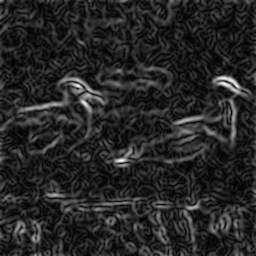}
    \caption{$\gamma$ (over-TV)\\from 20\% noise}
    \label{fig: gamma_gull_20}
    \end{subfigure}
    \begin{subfigure}[t]{0.15\textwidth}
    \centering
    \includegraphics[width=1\textwidth]{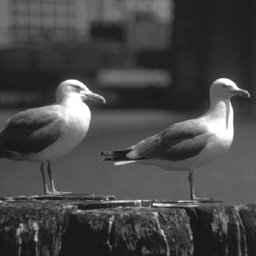}
    \caption{$\widehat{u}$ (GT)}
    \label{fig: gull_GT}
    \end{subfigure}
    \begin{subfigure}[t]{0.15\textwidth}
    \centering
    \includegraphics[width=1\textwidth]{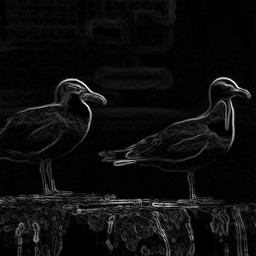}
    \caption{$\gamma$ from GT}
    \label{fig: gamma_gull_GT}
    \end{subfigure}
    \\
    \begin{subfigure}[t]{0.15\textwidth}
    \centering
    \includegraphics[width=1\textwidth]{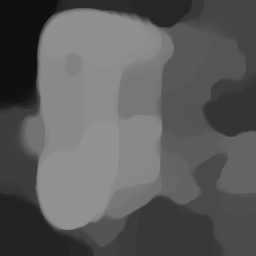}
    \caption{$\widehat{u}$ (over-TV)\\from 10\% noise}
    \label{fig: TV_fish_10}
    \end{subfigure}
    \begin{subfigure}[t]{0.15\textwidth}
    \centering
    \includegraphics[width=1\textwidth]{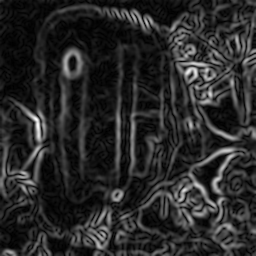}
    \caption{$\gamma$ (over-TV)\\from 10\% noise}
    \label{fig: gamma_fish_10}
    \end{subfigure}
    \begin{subfigure}[t]{0.15\textwidth}
    \centering
    \includegraphics[width=1\textwidth]{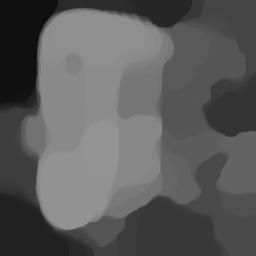}
    \caption{$\widehat{u}$ (over-TV)\\from 20\% noise}
    \label{fig: TV_fish_20}
    \end{subfigure}
    \begin{subfigure}[t]{0.15\textwidth}
    \centering
    \includegraphics[width=1\textwidth]{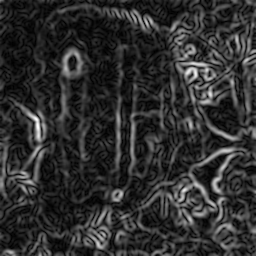}
    \caption{$\gamma$ (over-TV)\\from 20\% noise}
    \label{fig: gamma_fish_20}
    \end{subfigure}
    \begin{subfigure}[t]{0.15\textwidth}
    \centering
    \includegraphics[width=1\textwidth]{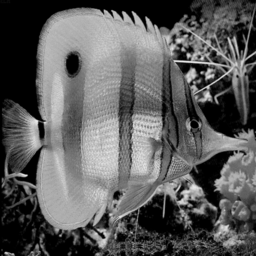}
    \caption{$\widehat{u}$ (GT)}
    \label{fig: fish_GT}
    \end{subfigure}
    \begin{subfigure}[t]{0.15\textwidth}
    \centering
    \includegraphics[width=1\textwidth]{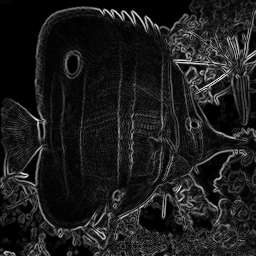}
    \caption{$\gamma$ from GT}
    \label{fig: gamma_fish_GT}
    \end{subfigure}
    \caption{Over-regularised $\TV$ solutions (\subref{fig: TV_buttefly_10},\subref{fig: TV_buttefly_20},\subref{fig: TV_gull_10},\subref{fig: TV_gull_20},\subref{fig: TV_fish_10} and \subref{fig: TV_fish_20}) and estimated $\gamma$ (\subref{fig: gamma_buttefly_10}, \subref{fig: gamma_buttefly_20}, \subref{fig: gamma_gull_10}, \subref{fig: gamma_gull_20}, \subref{fig: gamma_fish_10} and \subref{fig: gamma_fish_20}; rescaled for better visualisation) are compared with $\gamma$ obtained from the ground truth (ground truth shown in \subref{fig: butterfly_GT}, \subref{fig: gull_GT} and \subref{fig: fish_GT}; $\gamma$ shown in \subref{fig: gamma_buttefly_GT}, \subref{fig: gamma_gull_GT} and \subref{fig: gamma_fish_GT}).}
    \label{fig: computation of gamma}
\end{figure}

In Figures \ref{fig: gallery 10 noise} and \ref{fig: gallery 20 noise} we display  reconstructions obtained with Algorithm \ref{alg: PDHG} from images corrupted with $10\%$ or $20\%$ Gaussian noise, respectively. 
Total Variation (Figures~\ref{fig: butterly TV} -- \ref{fig: fish TV} for $10\%$ noise) produces the expected staircasing, which is significantly reduced with $\TVpwL$ (with $\gamma$ obtained using an overregularised $\TV$ reconstruction), as demonstrated in Figures~\ref{fig: butterfly TVpwL} -- \ref{fig: fish TVpwL}. Reconstructions obtained with $\TGV$ (Figures~\ref{fig: butterfly TGV} -- \ref{fig: fish TGV}) are slightly smoother; the values of $\SSIM$ and $\PSNR$ are sightly higher, but the computational time is up to an order a magnitude larger (cf., e.g., barbara, cameraman, fish, flowers). Supplied with a good a priori estimate of $\gamma$, $\TVpwL$ produces reconstructions that have much more details and a much smaller lost of contrast than other regularisers (Figures~\ref{fig: butterfly TVpwL GT} -- \ref{fig: fish TVpwL GT}).


\begin{figure}[tbhp]
    \centering
    \begin{subfigure}[t]{0.15\textwidth}
    \centering
    \includegraphics[width=1\textwidth]{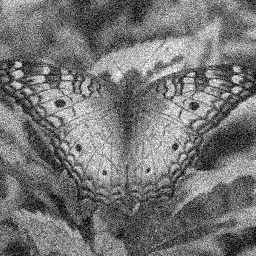}
    \caption{Noisy $f$}
    \end{subfigure}
    \begin{subfigure}[t]{0.15\textwidth}
    \centering
    \includegraphics[width=1\textwidth,trim=0px 128px 128px 0px, clip=true]{{./images_256/noise10/butterfly/u_noise}.png}
    \caption{Noisy $f$ (zoom)}
    \end{subfigure}
    \begin{subfigure}[t]{0.15\textwidth}
    \centering
    \includegraphics[width=1\textwidth]{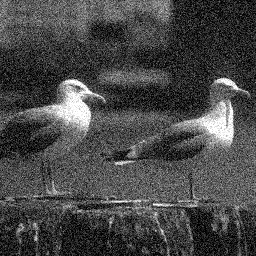}
    \caption{Noisy $f$}
    \end{subfigure}
    \begin{subfigure}[t]{0.15\textwidth}
    \centering
    \includegraphics[width=1\textwidth,trim=0px 128px 128px 0px, clip=true]{{./images_256/noise10/gull/u_noise}.png}
    \caption{Noisy $f$ (zoom)}
    \end{subfigure}
    \begin{subfigure}[t]{0.15\textwidth}
    \centering
    \includegraphics[width=1\textwidth]{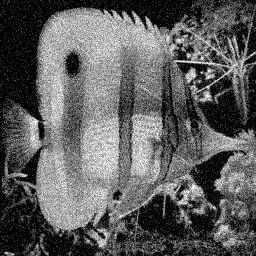}
    \caption{Noisy $f$}
    \end{subfigure}
    \begin{subfigure}[t]{0.15\textwidth}
    \centering
    \includegraphics[width=1\textwidth,trim=0px 128px 128px 0px, clip=true]{{./images_256/noise10/fish/u_noise}.png}
    \caption{Noisy $f$ (zoom)}
    \end{subfigure}
    \\
    \begin{subfigure}[t]{0.15\textwidth}
    \centering
    \includegraphics[width=1\textwidth]{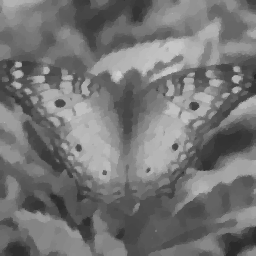}
    \caption{$\TV$}
    \label{fig: butterly TV}
    \end{subfigure}
    \begin{subfigure}[t]{0.15\textwidth}
    \centering
    \includegraphics[width=1\textwidth,trim=0px 128px 128px 0px, clip=true]{{./images_256/noise10/butterfly/u_TV_PDHG_SSIM0.7647_PSNR26.5521_cputime5.9}.png}
    \caption{$\TV$  (zoom)}
    \end{subfigure}
    \begin{subfigure}[t]{0.15\textwidth}
    \centering
    \includegraphics[width=1\textwidth]{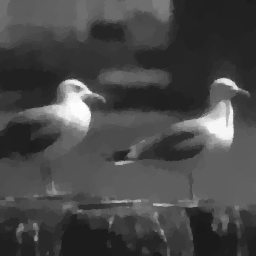}
    \caption{$\TV$}
    \end{subfigure}
    \begin{subfigure}[t]{0.15\textwidth}
    \centering
    \includegraphics[width=1\textwidth,trim=0px 128px 128px 0px, clip=true]{{./images_256/noise10/gull/u_TV_PDHG_SSIM0.84682_PSNR28.9895_cputime11.49}.png}
    \caption{$\TV$  (zoom)}
    \end{subfigure}
    \begin{subfigure}[t]{0.15\textwidth}
    \centering
    \includegraphics[width=1\textwidth]{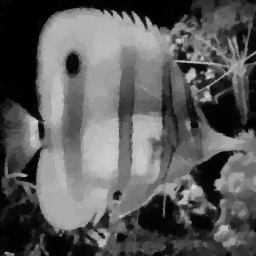}
    \caption{$\TV$}
    \end{subfigure}
    \begin{subfigure}[t]{0.15\textwidth}
    \centering
    \includegraphics[width=1\textwidth,trim=0px 128px 128px 0px, clip=true]{{./images_256/noise10/fish/u_TV_PDHG_SSIM0.72908_PSNR25.5025_cputime7.85}.png}
    \caption{$\TV$  (zoom)}
    \label{fig: fish TV}
    \end{subfigure}
    \\
    \begin{subfigure}[t]{0.15\textwidth}
    \centering
    \includegraphics[width=1\textwidth]{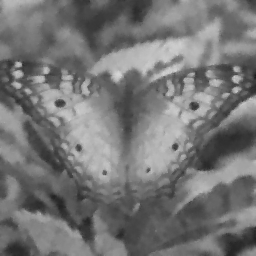}
    \caption{$\TVpwL$ (over-TV)}
    \label{fig: butterfly TVpwL}
    \end{subfigure}
    \begin{subfigure}[t]{0.15\textwidth}
    \centering
    \includegraphics[width=1\textwidth,trim=0px 128px 128px 0px, clip=true]{{./images_256/noise10/butterfly/u_TVpwL_PDHG_over_TV_SSIM0.7831_PSNR26.7345_cputime16.49}.png}
    \caption{$\TVpwL$\\(over-TV, zoom)}
    \end{subfigure}
    \begin{subfigure}[t]{0.15\textwidth}
    \centering
    \includegraphics[width=1\textwidth]{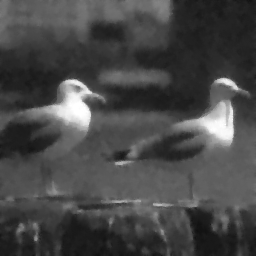}
    \caption{$\TVpwL$\\(over-TV)}
    \end{subfigure}
    \begin{subfigure}[t]{0.15\textwidth}
    \centering
    \includegraphics[width=1\textwidth,trim=0px 128px 128px 0px, clip=true]{{./images_256/noise10/gull/u_TVpwL_PDHG_over_TV_SSIM0.83929_PSNR28.6629_cputime17.17}.png}
    \caption{$\TVpwL$\\(over-TV, zoom)}
    \end{subfigure}
    \begin{subfigure}[t]{0.15\textwidth}
    \centering
    \includegraphics[width=1\textwidth]{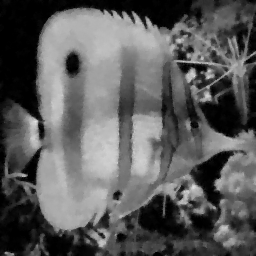}
    \caption{$\TVpwL$\\(over-TV)}
    \end{subfigure}
    \begin{subfigure}[t]{0.15\textwidth}
    \centering
    \includegraphics[width=1\textwidth,trim=0px 128px 128px 0px, clip=true]{{./images_256/noise10/fish/u_TVpwL_PDHG_over_TV_SSIM0.72052_PSNR25.4104_cputime14.69}.png}
    \caption{$\TVpwL$\\(over-TV, zoom)}
    \label{fig: fish TVpwL}
    \end{subfigure}
    \\
    \begin{subfigure}[t]{0.15\textwidth}
    \centering
    \includegraphics[width=1\textwidth]{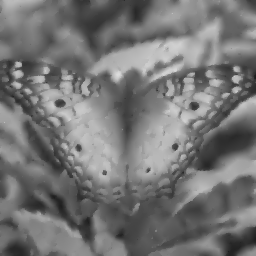}
    \caption{$\TGV^2$}
    \label{fig: butterfly TGV}
    \end{subfigure}
    \begin{subfigure}[t]{0.15\textwidth}
    \centering
    \includegraphics[width=1\textwidth,trim=0px 128px 128px 0px, clip=true]{{./images_256/noise10/butterfly/u_TGV_PDHG_SSIM0.8017_PSNR27.363_cputime82.48}.png}
    \caption{$\TGV^2$ (zoom)}
    \end{subfigure}
    \begin{subfigure}[t]{0.15\textwidth}
    \centering
    \includegraphics[width=1\textwidth]{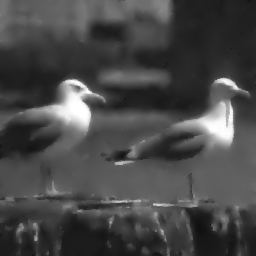}
    \caption{$\TGV^2$}
    \end{subfigure}
    \begin{subfigure}[t]{0.15\textwidth}
    \centering
    \includegraphics[width=1\textwidth,trim=0px 128px 128px 0px, clip=true]{{./images_256/noise10/gull/u_TGV_PDHG_SSIM0.8684_PSNR29.8027_cputime87.96}.png}
    \caption{$\TGV^2$ (zoom)}
    \end{subfigure}
     \begin{subfigure}[t]{0.15\textwidth}
    \centering
    \includegraphics[width=1\textwidth]{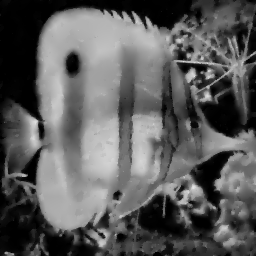}
    \caption{$\TGV^2$}
    \end{subfigure}
    \begin{subfigure}[t]{0.15\textwidth}
    \centering
    \includegraphics[width=1\textwidth,trim=0px 128px 128px 0px, clip=true]{{./images_256/noise10/fish/u_TGV_PDHG_SSIM0.73727_PSNR25.8609_cputime112.01}.png}
    \caption{$\TGV^2$ (zoom)}
    \label{fig: fish TGV}
    \end{subfigure}
    \\
    \begin{subfigure}[t]{0.15\textwidth}
    \centering
    \includegraphics[width=1\textwidth]{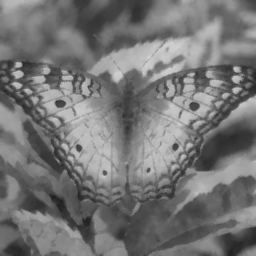}
    \caption{$\TVpwL$\\(GT)}
    \label{fig: butterfly TVpwL GT}
    \end{subfigure}
    \begin{subfigure}[t]{0.15\textwidth}
    \centering
    \includegraphics[width=1\textwidth,trim=0px 128px 128px 0px, clip=true]{{./images_256/noise10/butterfly/u_TVpwL_PDHG_GT_SSIM0.88752_PSNR29.462_cputime11.02}.png}
    \caption{$\TVpwL$\\(GT, zoom)}
    \end{subfigure}
    \begin{subfigure}[t]{0.15\textwidth}
    \centering
    \includegraphics[width=1\textwidth]{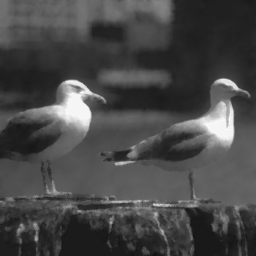}
    \caption{$\TVpwL$ (GT)}
    \end{subfigure}
    \begin{subfigure}[t]{0.15\textwidth}
    \centering
    \includegraphics[width=1\textwidth,trim=0px 128px 128px 0px, clip=true]{{./images_256/noise10/gull/u_TVpwL_PDHG_GT_SSIM0.92137_PSNR31.1986_cputime35.37}.png}
    \caption{$\TVpwL$\\(GT, zoom)}
    \end{subfigure}
    \begin{subfigure}[t]{0.15\textwidth}
    \centering
    \includegraphics[width=1\textwidth]{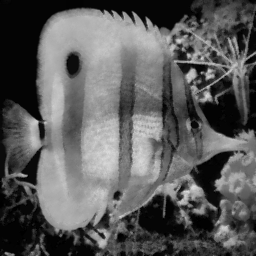}
    \caption{$\TVpwL$ (GT)}
    \end{subfigure}
    \begin{subfigure}[t]{0.15\textwidth}
    \centering
    \includegraphics[width=1\textwidth,trim=0px 128px 128px 0px, clip=true]{{./images_256/noise10/fish/u_TVpwL_PDHG_GT_SSIM0.76344_PSNR26.8469_cputime69.49}.png}
    \caption{$\TVpwL$\\(GT, zoom)}
    \label{fig: fish TVpwL GT}
    \end{subfigure}
    \caption{The \texttt{butterfly}, \texttt{gull} and the \texttt{fish} images corrupted with 10\% of Gaussian noise and denoised using $\TV$ (second row), $\TGV$ (forth row) and $\TVpwL^\gamma$ with different $\gamma$ (third and fourth rows). $\TV$ produces characteristic staircasing, which is no longer present in the much smoother $\TGV$ reconstructions. $\TVpwL^\gamma$ with $\gamma$ estimated from the noisy image is somewhere between $\TV$ and $\TGV$: there is no staircasing, but the images are not as smooth as $\TGV$. With $\gamma$ estimated from the ground truth, $\TVpwL$ produces almost perfect reconstructions. We include these images to demonstrate what performance $\TVpwL$ can theoretically achieve if supplied with a good parameter $\gamma$. We also emphasise that $\gamma$ only contains information about the magnitude of the gradient, not its direction.}
    \label{fig: gallery 10 noise}
\end{figure}

\begin{figure}[tbhp]
    \centering
    \begin{subfigure}[t]{0.15\textwidth}
    \centering
    \includegraphics[width=1\textwidth]{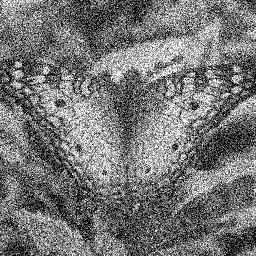}
    \caption{Noisy $f$}
    \end{subfigure}
    \begin{subfigure}[t]{0.15\textwidth}
    \centering
    \includegraphics[width=1\textwidth,trim=0px 128px 128px 0px, clip=true]{{./images_256/noise20/butterfly/u_noise}.png}
    \caption{Noisy $f$ (zoom)}
    \end{subfigure}
    \begin{subfigure}[t]{0.15\textwidth}
    \centering
    \includegraphics[width=1\textwidth]{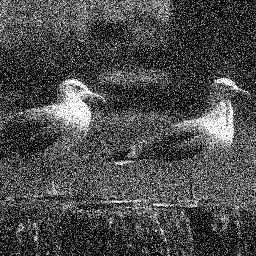}
    \caption{Noisy $f$}
    \end{subfigure}
    \begin{subfigure}[t]{0.15\textwidth}
    \centering
    \includegraphics[width=1\textwidth,trim=0px 128px 128px 0px, clip=true]{{./images_256/noise20/gull/u_noise}.png}
    \caption{Noisy $f$ (zoom)}
    \end{subfigure}
    \begin{subfigure}[t]{0.15\textwidth}
    \centering
    \includegraphics[width=1\textwidth]{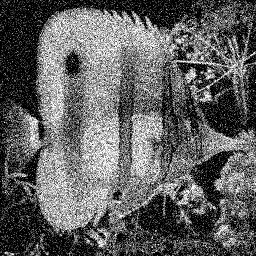}
    \caption{Noisy $f$}
    \end{subfigure}
    \begin{subfigure}[t]{0.15\textwidth}
    \centering
    \includegraphics[width=1\textwidth,trim=0px 128px 128px 0px, clip=true]{{./images_256/noise10/fish/u_noise}.png}
    \caption{Noisy $f$ (zoom)}
    \end{subfigure}
    \\
    \begin{subfigure}[t]{0.15\textwidth}
    \centering
    \includegraphics[width=1\textwidth]{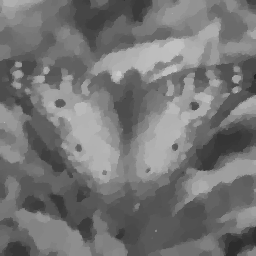}
    \caption{$\TV$}
    \end{subfigure}
    \begin{subfigure}[t]{0.15\textwidth}
    \centering
    \includegraphics[width=1\textwidth,trim=0px 128px 128px 0px, clip=true]{{./images_256/noise20/butterfly/u_TV_PDHG_SSIM0.64412_PSNR23.805_cputime16.99}.png}
    \caption{$\TV$  (zoom)}
    \end{subfigure}
    \begin{subfigure}[t]{0.15\textwidth}
    \centering
    \includegraphics[width=1\textwidth]{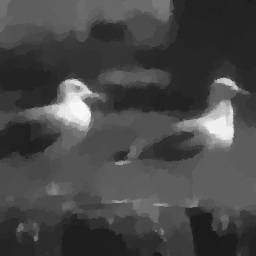}
    \caption{$\TV$}
    \end{subfigure}
    \begin{subfigure}[t]{0.15\textwidth}
    \centering
    \includegraphics[width=1\textwidth,trim=0px 128px 128px 0px, clip=true]{{./images_256/noise20/gull/u_TV_PDHG_SSIM0.77671_PSNR26.1191_cputime16.8}.png}
    \caption{$\TV$  (zoom)}
    \end{subfigure}
    \begin{subfigure}[t]{0.15\textwidth}
    \centering
    \includegraphics[width=1\textwidth]{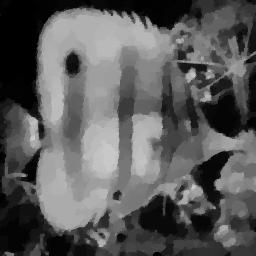}
    \caption{$\TV$}
    \end{subfigure}
    \begin{subfigure}[t]{0.15\textwidth}
    \centering
    \includegraphics[width=1\textwidth,trim=0px 128px 128px 0px, clip=true]{{./images_256/noise20/fish/u_TV_PDHG_SSIM0.58593_PSNR22.4736_cputime16.9}.png}
    \caption{$\TV$  (zoom)}
    \end{subfigure}
    \\
    \begin{subfigure}[t]{0.15\textwidth}
    \centering
    \includegraphics[width=1\textwidth]{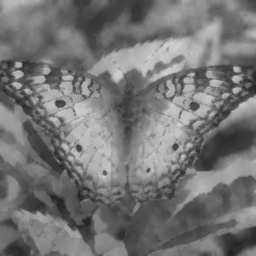}
    \caption{$\TVpwL$\\(GT)}
    \end{subfigure}
    \begin{subfigure}[t]{0.15\textwidth}
    \centering
    \includegraphics[width=1\textwidth,trim=0px 128px 128px 0px, clip=true]{{./images_256/noise20/butterfly/u_TVpwL_PDHG_GT_SSIM0.82569_PSNR27.008_cputime17.07}.png}
    \caption{$\TVpwL$\\(GT, zoom)}
    \end{subfigure}
    \begin{subfigure}[t]{0.15\textwidth}
    \centering
    \includegraphics[width=1\textwidth]{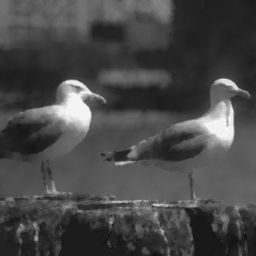}
    \caption{$\TVpwL$ (GT)}
    \end{subfigure}
    \begin{subfigure}[t]{0.15\textwidth}
    \centering
    \includegraphics[width=1\textwidth,trim=0px 128px 128px 0px, clip=true]{{./images_256/noise20/gull/u_TVpwL_PDHG_GT_SSIM0.88377_PSNR29.1533_cputime33.91}.png}
    \caption{$\TVpwL$\\(GT, zoom)}
    \end{subfigure}
    \begin{subfigure}[t]{0.15\textwidth}
    \centering
    \includegraphics[width=1\textwidth]{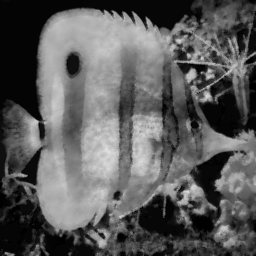}
    \caption{$\TVpwL$ (GT)}
    \end{subfigure}
    \begin{subfigure}[t]{0.15\textwidth}
    \centering
    \includegraphics[width=1\textwidth,trim=0px 128px 128px 0px, clip=true]{{./images_256/noise20/fish/u_TVpwL_PDHG_GT_SSIM0.68699_PSNR24.8756_cputime73.11}.png}
    \caption{$\TVpwL$\\(GT, zoom)}
    \end{subfigure}
    \\
    \begin{subfigure}[t]{0.15\textwidth}
    \centering
    \includegraphics[width=1\textwidth]{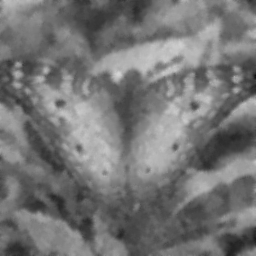}
    \caption{$\TVpwL$ (over-TV)}
    \end{subfigure}
    \begin{subfigure}[t]{0.15\textwidth}
    \centering
    \includegraphics[width=1\textwidth,trim=0px 128px 128px 0px, clip=true]{{./images_256/noise20/butterfly/u_TVpwL_PDHG_over_TV_SSIM0.67273_PSNR24.0509_cputime38.13}.png}
    \caption{$\TVpwL$\\(over-TV, zoom)}
    \end{subfigure}
    \begin{subfigure}[t]{0.15\textwidth}
    \centering
    \includegraphics[width=1\textwidth]{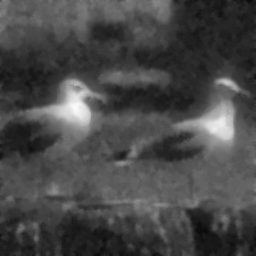}
    \caption{$\TVpwL$\\(over-TV)}
    \end{subfigure}
    \begin{subfigure}[t]{0.15\textwidth}
    \centering
    \includegraphics[width=1\textwidth,trim=0px 128px 128px 0px, clip=true]{{./images_256/noise20/gull/u_TVpwL_PDHG_over_TV_SSIM0.73533_PSNR24.7523_cputime70.19}.png}
    \caption{$\TVpwL$\\(over-TV, zoom)}
    \end{subfigure}
    \begin{subfigure}[t]{0.15\textwidth}
    \centering
    \includegraphics[width=1\textwidth]{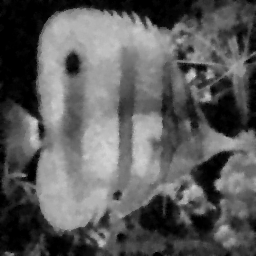}
    \caption{$\TVpwL$\\(over-TV)}
    \end{subfigure}
    \begin{subfigure}[t]{0.15\textwidth}
    \centering
    \includegraphics[width=1\textwidth,trim=0px 128px 128px 0px, clip=true]{{./images_256/noise20/fish/u_TVpwL_PDHG_over_TV_SSIM0.57187_PSNR22.3624_cputime32.75}.png}
    \caption{$\TVpwL$\\(over-TV, zoom)}
    \end{subfigure}
    \\
    \begin{subfigure}[t]{0.15\textwidth}
    \centering
    \includegraphics[width=1\textwidth]{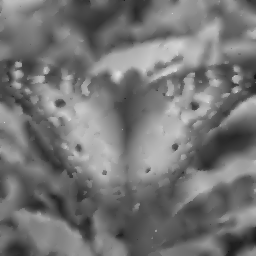}
    \caption{$\TGV^2$}
    \end{subfigure}
    \begin{subfigure}[t]{0.15\textwidth}
    \centering
    \includegraphics[width=1\textwidth,trim=0px 128px 128px 0px, clip=true]{{./images_256/noise20/butterfly/u_TGV_PDHG_SSIM0.68934_PSNR24.5645_cputime111.42}.png}
    \caption{$\TGV^2$ (zoom)}
    \end{subfigure}
    \begin{subfigure}[t]{0.15\textwidth}
    \centering
    \includegraphics[width=1\textwidth]{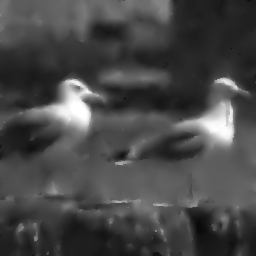}
    \caption{$\TGV^2$}
    \end{subfigure}
    \begin{subfigure}[t]{0.15\textwidth}
    \centering
    \includegraphics[width=1\textwidth,trim=0px 128px 128px 0px, clip=true]{{./images_256/noise20/gull/u_TGV_PDHG_SSIM0.80033_PSNR26.8687_cputime120.63}.png}
    \caption{$\TGV^2$ (zoom)}
    \end{subfigure}
     \begin{subfigure}[t]{0.15\textwidth}
    \centering
    \includegraphics[width=1\textwidth]{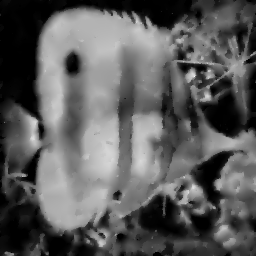}
    \caption{$\TGV^2$}
    \end{subfigure}
    \begin{subfigure}[t]{0.15\textwidth}
    \centering
    \includegraphics[width=1\textwidth,trim=0px 128px 128px 0px, clip=true]{{./images_256/noise20/fish/u_TGV_PDHG_SSIM0.59631_PSNR22.8763_cputime144.24}.png}
    \caption{$\TGV^2$ (zoom)}
    \end{subfigure}
    \caption{The \texttt{butterfly}, \texttt{gull} and the \texttt{fish} images corrupted with 20\% of Gaussian noise and denoised using $\TV$ (second row), $\TGV$ (forth row) and $\TVpwL^\gamma$ with different $\gamma$ (third and fourth rows). The results are qualitatively the same as with $10\%$ noise (Figure~\ref{fig: gallery 10 noise}). $\TV$ produces characteristic staircasing, which is no longer present in the much smoother $\TGV$ reconstructions. $\TVpwL^\gamma$ with $\gamma$ estimated from the noisy image is somewhere between $\TV$ and $\TGV$: there is no staircasing, but the images are not as smooth as $\TGV$. With $\gamma$ estimated from the ground truth, $\TVpwL$ produces almost perfect reconstructions. We include these images to demonstrate what performance $\TVpwL$ can theoretically achieve if supplied with a good parameter $\gamma$. We also emphasise that $\gamma$ only contains information about the magnitude of the gradient, not its direction.}
    \label{fig: gallery 20 noise}
\end{figure}

The results obtained with CVX demonstrate the same qualitative behaviour (Table~\ref{tab: quantitative results 10 Gaussian}). The reconstructions are almost identical to those obtained with the primal dual scheme and are not shown here.

To investigate the effect of the regularisation parameter $\lambda$ in~\eqref{eq: over-TV argmin} that controls the amount of $\TV$-over\-regularisation used to estimate $\gamma$, we perform experiments with $\lambda = 100;200;300$ and $400$ on the \texttt{butterfly} image (with $10\%$ noise). The results are shown in Figure~\ref{fig: comp_lambda}. Surprisingly, although the overregularised $\TV$ solutions differ significantly (Figures~\ref{fig: TV_lambda_100},~\ref{fig: TV_lambda_200},~\ref{fig: TV_lambda_300} and~\ref{fig: TV_lambda_400}) and there is visible difference in the estimated $\gamma$ (Figures~\ref{fig: gamma_lambda_100},~\ref{fig: gamma_lambda_200},~\ref{fig: gamma_lambda_300} and~\ref{fig: gamma_lambda_400}), the corresponding $\TVpwL$ reconstructions differ only marginally, which is also confirmed by the very similar $\SSIM$ and $\PSNR$ values (Figures~\ref{fig: TVpwl_lambda_100},~\ref{fig: TVpwl_lambda_200},~\ref{fig: TVpwl_lambda_300} and~\ref{fig: TVpwl_lambda_400}).

\begin{figure}[tbhp]
    \centering
    \begin{subfigure}[t]{0.15\textwidth}
    \centering
    \includegraphics[width=1\textwidth]{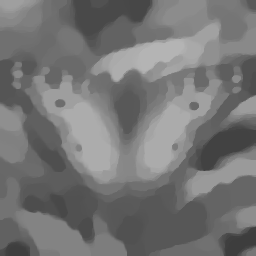}
    \caption{$\TV$, $\lambda=100$}
    \label{fig: TV_lambda_100}
    \end{subfigure}
    \begin{subfigure}[t]{0.15\textwidth}
    \centering
    \includegraphics[width=1\textwidth]{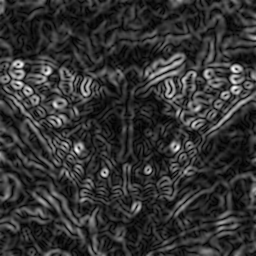}
    \caption{estimated $\gamma$\newline with $\lambda = 100$}
    \label{fig: gamma_lambda_100}
    \end{subfigure}    \begin{subfigure}[t]{0.15\textwidth}
    \centering
    \includegraphics[width=1\textwidth]{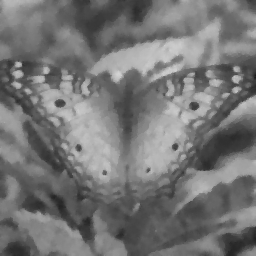}
    \caption{$\TVpwL^\gamma$, $\lambda=100$\newline $\SSIM = 0.781$, $\PSNR = 26.68$}
    \label{fig: TVpwl_lambda_100}
    \end{subfigure}
    \begin{subfigure}[t]{0.15\textwidth}
    \centering
    \includegraphics[width=1\textwidth]{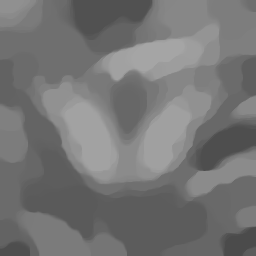}
    \caption{$\TV$, $\lambda=200$}
    \label{fig: TV_lambda_200}
    \end{subfigure}
    \begin{subfigure}[t]{0.15\textwidth}
    \centering
    \includegraphics[width=1\textwidth]{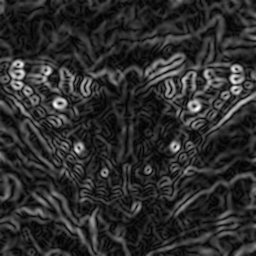}
    \caption{estimated $\gamma$\newline with $\lambda = 200$}
    \label{fig: gamma_lambda_200}
    \end{subfigure}    \begin{subfigure}[t]{0.15\textwidth}
    \centering
    \includegraphics[width=1\textwidth]{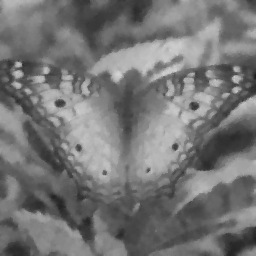}
    \caption{$\TVpwL^\gamma$, $\lambda=200$\\$\SSIM = 0.782$, $\PSNR = 26.71$}
    \label{fig: TVpwl_lambda_200}
    \end{subfigure}
    \\
    \begin{subfigure}[t]{0.15\textwidth}
    \centering
    \includegraphics[width=1\textwidth]{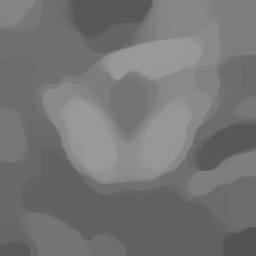}
    \caption{$\TV$, $\lambda = 300$}
    \label{fig: TV_lambda_300}
    \end{subfigure}
    \begin{subfigure}[t]{0.15\textwidth}
    \centering
    \includegraphics[width=1\textwidth]{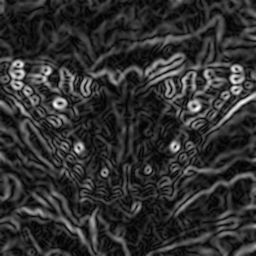}
    \caption{estimated $\gamma$\\with $\lambda = 300$}
    \label{fig: gamma_lambda_300}
    \end{subfigure}    \begin{subfigure}[t]{0.15\textwidth}
    \centering
    \includegraphics[width=1\textwidth]{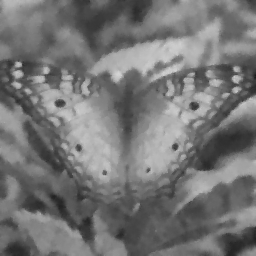}
    \caption{$\TVpwL^\gamma$, $\lambda = 300$ $\SSIM = 0.783$, $\PSNR = 26.72$}
    \label{fig: TVpwl_lambda_300}
    \end{subfigure}    
    \begin{subfigure}[t]{0.15\textwidth}
    \centering
    \includegraphics[width=1\textwidth]{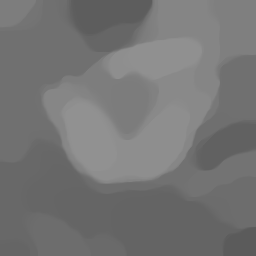}
    \caption{$\TV$, $\lambda = 400$}
    \label{fig: TV_lambda_400}
    \end{subfigure}
    \begin{subfigure}[t]{0.15\textwidth}
    \centering
    \includegraphics[width=1\textwidth]{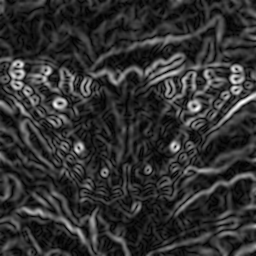}
    \caption{estimated $\gamma$\newline with $\lambda = 400$}
    \label{fig: gamma_lambda_400}
    \end{subfigure}    \begin{subfigure}[t]{0.15\textwidth}
    \centering
    \includegraphics[width=1\textwidth]{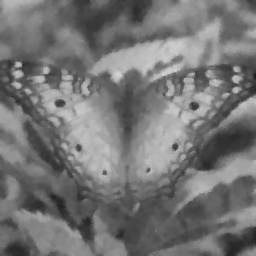}
    \caption{$\TVpwL^\gamma$, $\lambda = 400$\newline $\SSIM = 0.783$, $\PSNR = 26.73$}
    \label{fig: TVpwl_lambda_400}
    \end{subfigure} 
    \caption{Although the over-regularised $\TV$ solutions differ significantly (\subref{fig: TV_lambda_100},\subref{fig: TV_lambda_200},\subref{fig: TV_lambda_300} and \subref{fig: TV_lambda_400}) and the estimated $\gamma$ are also different (\subref{fig: gamma_lambda_100},\subref{fig: gamma_lambda_200},\subref{fig: gamma_lambda_300} and \subref{fig: gamma_lambda_400}), the corresponding $\TVpwL$ reconstructions are almost identical (\subref{fig: TVpwl_lambda_100},\subref{fig: TVpwl_lambda_200},\subref{fig: TVpwl_lambda_300} and \subref{fig: TVpwl_lambda_400}) and the $\SSIM$ and $\PSNR$ values are very similar.}
    \label{fig: comp_lambda}
\end{figure}


\begin{table}[tbhp]\small
\caption{Results for our dataset of grey-scale images in Figure \ref{fig: dataset}, corrupted with 10\% of Gaussian noise and the PDHG Algorithm \ref{alg: PDHG} (CVX results in brackets). The runtime for $\TVpwL$ is up to an order of magnitude smaller than for $\TGV$ (e.g., in \texttt{barbara, cameramen, fish, flowers}) and on the same scale as for $\TV$ (typically $1.5-2$ times larger). The reconstruction quality is similar to $\TGV$. $\TVpwL$ with $\gamma$ estimated from the ground truth consistently obtains the best results with a wide margin (although, of course, this is an idealistic situation but it shows nevertheless the potential of $\TVpwL$ with a better estimate of $\gamma$).  The $\SSIM$ and $\PSNR$ measures do not always reflect the visual results in Figure~\ref{fig: gallery 10 noise}; for instance, $\TV$ sometimes obtains similar values of $\SSIM$ as $\TGV$ and $\TVpwL$ despite visible staircasing (e.g., in \texttt{fish}).}
\label{tab: quantitative results 10 Gaussian}
\begin{tabularx}{\textwidth}{X|c|c|cc|c}
\toprule
\multirow{2}{*}{Image} & \multirow{2}{*}{Index} & \multirow{2}{*}{$\TV$} &  \multicolumn{2}{c|}{$\TVpwL$} & \multirow{2}{*}{$\TGV^2$} \\
 &  &  & (GT) & (over-$\TV$) & \\
\midrule
                   & SSIM          & 0.779 (0.779) & 0.860 (0.853)  & 0.782 (0.782)  & 0.800 (0.800)\\
\texttt{barbara}   & PSNR          & 27.01 (27.01) & 29.26 (28.57)  & 27.05 (27.04)  & 27.79 (27.79)\\
                   & cputime (s.)  & 09.49 (95.25) & 13.76 (167.13) & 17.02 (161.93) & 104.01 (199.27)\\
\midrule
                   & SSIM          & 0.581 (0.582) & 0.742 (0.706)  & 0.575 (0.574)  & 0.593 (0.590)\\
\texttt{brickwall} & PSNR          & 25.49 (25.50) & 27.09 (26.76)  & 25.44 (25.44)  & 25.57 (25.58)\\
                   & cputime (s.)  & 05.72 (94.57) & 11.12 (161.00) & 13.00 (163.85) & 69.91 (196.94)\\
\midrule
                   & SSIM          & 0.765 (0.765) & 0.888 (0.869)  & 0.783 (0.783)  & 0.802 (0.801)\\
\texttt{butterfly} & PSNR          & 26.55 (26.55) & 29.46 (28.50)  & 26.73 (26.73)  & 27.36 (27.35)\\
                   & cputime (s.)  & 05.90 (97.97) & 11.02 (162.52) & 16.49 (164.67) & 82.48 (205.17)\\
\midrule
                   & SSIM          & 0.805 (0.805) & 0.845 (0.845)  & 0.788 (0.788)  & 0.802 (0.801)\\
\texttt{cameraman} & PSNR          & 27.32 (27.33) & 27.29 (27.28)  & 26.78 (26.77)  & 27.32 (27.32)\\
                   & cputime (s.)  & 07.57 (95.45) & 22.45 (164.38) & 13.81 (160.92) & 108.16 (197.95)\\
\midrule
                   & SSIM          & 0.729 (0.731) & 0.763 (0.749)  & 0.721 (0.712)  & 0.737 (0.751)\\
\texttt{fish}      & PSNR          & 25.50 (25.51) & 26.85 (26.67)  & 25.41 (25.43)  & 25.86 (25.89)\\
                   & cputime (s.)  & 07.85 (96.12) & 69.49 (173.02) & 14.69 (163.87) & 112.01 (204.94)\\
\midrule
                   & SSIM          & 0.787 (0.787) & 0.844 (0.844)  & 0.786 (0.786)  & 0.792 (0.792)\\
\texttt{flowers  } & PSNR          & 22.18 (22.18) & 22.93 (22.93)  & 22.14 (22.14)  & 22.26 (22.26)\\
                   & cputime (s.)  & 06.12 (94.72) & 23.59 (161.68) & 12.36 (159.55) & 129.16 (201.21)\\
\midrule
                   & SSIM          & 0.847 (0.847) & 0.921 (0.916)  & 0.839 (0.839) & 0.868 (0.868)\\
\texttt{gull}      & PSNR          & 28.99 (28.99) & 31.20 (30.59)  & 28.66 (28.66) & 29.80 (29.79)\\
                   & cputime (s.)  & 11.49 (98.86) & 35.37 (169.48) & 17.17 (172.26) & 87.96 (201.47)\\
\midrule
                   & SSIM          & 0.649 (0.649) & 0.744 (0.734)  & 0.655 (0.655)  & 0.658 (0.659)\\
\texttt{house}     & PSNR          & 26.11 (26.11) & 27.07 (26.88)  & 26.04 (26.04)  & 26.19 (26.19)\\
                   & cputime (s.)  & 06.27 (95.30) & 13.01 (164.21) & 13.07 (160.93) & 82.55 (201.01)\\
\midrule
                   & SSIM          & 0.667 (0.667) & 0.808 (0.772) & 0.681 (0.681) & 0.688 (0.687)\\
\texttt{owl}       & PSNR          & 25.66 (25.66) & 27.80 (26.91) & 25.81 (25.80) & 26.03 (26.02)\\
                   & cputime (s.)  & 05.27 (98.60) & 07.06 (164.18) & 10.18 (164.95) & 87.76 (208.66)\\
\midrule
                    & SSIM         & 0.792 (0.792) & 0.864 (0.864) & 0.792 (0.797) & 0.811 (0.820)\\
\texttt{pine\_tree} & PSNR         & 25.88 (25.89) & 26.94 (26.93) & 25.83 (25.83) & 26.38 (26.41)\\
                    & cputime (s.) & 07.46 (94.86) & 27.58 (164.70) & 15.21 (163.82) & 102.93 (202.43)\\
\midrule
                    & SSIM         & 0.713 (0.713) & 0.820 (0.808) & 0.730 (0.730) & 0.745 (0.744)\\
\texttt{squirrel}   & PSNR         & 27.23 (27.22) & 28.96 (28.41) & 27.45 (27.45) & 27.98 (27.96)\\
                    & cputime (s.) & 08.06 (95.04) & 16.85 (167.99) & 15.08 (162.05) & 79.45 (198.86)\\
\bottomrule
\end{tabularx}
\end{table}

\begin{table}[tbhp]\small
\caption{Results for our dataset of grey-scale images in Figure \ref{fig: dataset}, corrupted with 20\% of Gaussian noise and the PDHG Algorithm \ref{alg: PDHG} (CVX results in brackets). The results are qualitatively the same as for $10\%$ noise (Table~\ref{tab: quantitative results 10 Gaussian}). The runtime for $\TVpwL$ is still significantly smaller than for $\TGV$ (e.g., in \texttt{cameraman, fish, flowers}) and on the same scale as for $\TV$ (typically $2-2.5$ times larger). The reconstruction quality is similar to $\TGV$ and in a few cases even slightly better (\texttt{brickwall, owl}). $\TVpwL$ with $\gamma$ estimated from the ground truth consistently obtains the best results with a wide margin (although, of course, this is an idealistic situation but it shows nevertheless the potential of $\TVpwL$ with a better estimate of $\gamma$). The $\SSIM$ and $\PSNR$ measures do not always reflect the visual results in Figure~\ref{fig: gallery 20 noise}; for instance, $\TV$ sometimes obtains similar results as $\TGV$ and $\TVpwL$ despite visible staircasing (e.g., in \texttt{fish}).}
\label{tab: quantitative results 20 Gaussian}
\begin{tabularx}{\textwidth}{X|c|c|cc|c}
\toprule
\multirow{2}{*}{Image} & \multirow{2}{*}{Index} & \multirow{2}{*}{$\TV$} &  \multicolumn{2}{c|}{$\TVpwL$} & \multirow{2}{*}{$\TGV^2$} \\
 &  &  & (GT) & (over-$\TV$) & \\
\midrule
                   & SSIM          & 0.679 (0.679) & 0.809 (0.788)  & 0.681 (0.681)  & 0.704 (0.703)\\
\texttt{barbara}   & PSNR          & 24.05 (24.04) & 27.06 (25.44)  & 24.13 (24.12)  & 24.99 (24.98)\\
                   & cputime (s.)  & 17.46 (95.18) & 18.99 (161.55) & 42.77 (165.00) & 127.65 (199.06)\\
\midrule
                   & SSIM          & 0.373 (0.375) & 0.614 (0.548)  & 0.395 (0.395)  & 0.383 (0.388)\\
\texttt{brickwall} & PSNR          & 23.48 (23.48) & 24.98 (23.92)  & 23.59 (25.59)  & 23.48 (23.49)\\
                   & cputime (s.)  & 11.91 (94.57) & 19.31 (160.73) & 25.65 (163.06) & 105.44 (211.26)\\
\midrule
                   & SSIM          & 0.644 (0.644) & 0.826 (0.783)  & 0.673 (0.673)  & 0.689 (0.688)\\
\texttt{butterfly} & PSNR          & 23.81 (23.80) & 27.00 (25.14)  & 24.05 (24.04)  & 24.56 (24.55)\\
                   & cputime (s.)  & 16.99 (94.52) & 17.07 (161.55) & 38.13 (168.36) & 111.42 (201.35)\\
\midrule
                   & SSIM          & 0.731 (0.731) & 0.789 (0.795)  & 0.666 (0.667)   & 0.713 (0.714)\\
\texttt{cameraman} & PSNR          & 24.29 (24.30) & 25.30 (25.17)  & 23.26 (23.26)   & 24.15 (24.17)\\
                   & cputime (s.)  & 13.47 (95.98) & 32.06 (169.86) & 31.16 (163.18) & 137.99 (202.67)\\
\midrule
                   & SSIM          & 0.586 (0.588) & 0.687 (0.638)  & 0.572 (0.563)  & 0.596 (0.622)\\
\texttt{fish}      & PSNR          & 22.47 (22.48) & 24.88 (23.70)  & 22.36 (22.37)  & 22.88 (22.92)\\
                   & cputime (s.)  & 16.90 (97.83) & 73.11 (162.81) & 32.75 (163.12) & 144.24 (202.12)\\
\midrule
                   & SSIM          & 0.585 (0.585) & 0.756 (0.698)  & 0.592 (0.592)  & 0.596 (0.596)\\
\texttt{flowers  } & PSNR          & 18.99 (18.99) & 20.65 (20.07)  & 19.00 (19.00)  & 19.08 (19.08)\\
                   & cputime (s.)  & 13.19 (96.63) & 17.00 (162.70) & 26.96 (159.65) & 153.47 (198.14)\\
\midrule
                   & SSIM          & 0.777 (0.777) & 0.884 (0.872)  & 0.735 (0.736)  & 0.800 (0.799)\\
\texttt{gull}      & PSNR          & 26.12 (26.12) & 29.15 (27.45)  & 24.75 (24.74)  & 26.87 (26.85)\\
                   & cputime (s.)  & 16.80 (95.79) & 33.91 (163.39) & 70.19 (170.17) & 120.63 (216.58)\\
\midrule
                   & SSIM          & 0.527 (0.527) & 0.649 (0.626)  & 0.533 (0.533)  & 0.536 (0.537)\\
\texttt{house}     & PSNR          & 23.80 (23.80) & 25.10 (24.23)  & 23.60 (23.60)  & 23.88 (23.88)\\
                   & cputime (s.)  & 15.31 (94.71) & 17.35 (161.57) & 32.02 (163.97) & 113.33 (204.15)\\
\midrule
                   & SSIM          & 0.515 (0.515) & 0.705 (0.648)  & 0.546 (0.546)  & 0.544 (0.544)\\
\texttt{owl}       & PSNR          & 23.14 (23.14) & 25.33 (23.64)  & 23.36 (23.35)  & 23.64 (23.63)\\
                   & cputime (s.)  & 15.63 (95.01) & 16.75 (159.07) & 34.68 (163.69) & 114.92 (207.76)\\
\midrule
                    & SSIM         & 0.673 (0.673) & 0.806 (0.765)  & 0.656 (0.670)  & 0.683 (0.707)\\
\texttt{pine\_tree} & PSNR         & 23.22 (23.22) & 25.10 (24.25)  & 23.08 (23.08)  & 23.65 (23.69)\\
                    & cputime (s.) & 15.65 (95.19) & 24.55 (163.32) & 35.76 (161.17) & 143.66 (201.69)\\
\midrule
                    & SSIM         & 0.626 (0.626) & 0.750 (0.733)  & 0.643 (0.643)  & 0.668 (0.667)\\
\texttt{squirrel}   & PSNR         & 24.74 (24.73) & 26.89 (25.54)  & 24.90 (24.89)  & 25.84 (25.83)\\
                    & cputime (s.) & 17.79 (98.50) & 22.96 (164.53) & 48.56 (167.29) & 111.22 (203.91)\\
\bottomrule
\end{tabularx}
\end{table}

\section{Conclusion}
In this paper we have analysed a first order $\TV$ type regulariser that contains in its kernel all functions with a given (possibly, space dependant) Lipschitz constant and therefore only penalises gradients above a certain predefined threshold. From the theoretical point of view, its properties are similar to those of Total Variation (e.g., both obey a maximum principle). From the numerical point of view, their performance is different; the proposed regulariser significantly reduces staircasing while requiring roughly the same computational time as Total Variation. Compared with  Total Generalised Variation, which is a second order regulariser, the proposed regulariser can be up to an order of magnitude faster.

The performance of the proposed regulariser significantly depends on the suitability of the spatially varying Lipschitz constant $\gamma$ that defines the amount of variation allowed in the reconstruction without any penalty. If a good estimate is available, the results can be much better than with other regularisers. 

Ways of finding a good $\gamma$, however, are beyond the scope of this paper, where we rather concentrate on theoretical properties and efficient numerical methods in the case when $\gamma$ is given. 
We mention, however, that one possible way of estimating $\gamma$ from a noisy image is using a cartoon-texture decomposition such as in \cite{BuadesLeMorelVese2011,LeGuen2014,BuadesLisani2016}.

\paragraph{Acknowledgements.}
This work has been supported by the European Union’s Horizon 2020 research and
innovation programme under the Marie Sklodowska-Curie grant agreement No. 777826
(NoMADS). 
YK is supported by the Royal Society (Newton International Fellowship NF170045 Quantifying Uncertainty in Model-Based Data Inference Using Partial Order) and the Cantab Capital Institute for the Mathematics of Information. 
SP and CBS acknowledge support from the Leverhulme Trust project on Unveiling the Invisible: Mathematics for conservation in Arts and Humanities. 
CBS also acknowledges support from the Philip Leverhulme Prize, the EPSRC grants No.\ EP/T003553/1, EP/S026045/1 and EP/P020259/1, the EPSRC Centre No.\ EP/N014588/1, the European Union Horizon 2020 research
and innovation programmes under the Marie Sk{\l}odowska-Curie grant agreement No.\ 691070 CHiPS, the Cantab Capital Institute for the Mathematics of Information and the Alan Turing Institute.  We also acknowledge the support of NVIDIA Corporation with the donation of a Quadro P6000 and a Titan Xp GPUs.

\printbibliography

\end{document}